\documentclass[a4paper,11pt]{article}

\usepackage[top=1in, bottom=1.25in, left=1.25in, right=1.25in]{geometry}
\usepackage{amsfonts,amssymb,amsmath,amsthm} 
\usepackage{srcltx,tabularx}
\usepackage[utf8]{inputenc} 
\usepackage[cyr]{aeguill}
\usepackage[alphabetic, nobysame]{amsrefs}
\usepackage{enumerate}
\usepackage[english]{babel}
\usepackage{authblk}
\usepackage[hidelinks,pagebackref=false,linktoc=page]{hyperref}
\usepackage{graphicx}
\usepackage{pstricks,pst-plot,pst-node}
\usepackage[bottom]{footmisc}
\usepackage[all]{hypcap}
\usepackage[toc,page]{appendix}
\usepackage{breqn}
\usepackage{arydshln}
\usepackage{stfloats}
\usepackage{mathtools}
\usepackage{caption}
\usepackage{subcaption}

\newtheorem{innerprimetheorem}{Theorem}
\newenvironment{primetheorem}[1]
  {\renewcommand\theinnerprimetheorem{\ref*{#1}$'$}\innerprimetheorem}
  {\endinnerprimetheorem}

\newtheorem{innerprimecorollary}{Corollary}
\newenvironment{primecorollary}[1]
  {\renewcommand\theinnerprimecorollary{\ref*{#1}$'$}\innerprimecorollary}
  {\endinnerprimecorollary}

\newtheorem{theorem}{Theorem}[section]
\newtheorem{lemma}[theorem]{Lemma}
\newtheorem{proposition}[theorem]{Proposition}
\newtheorem{corollary}[theorem]{Corollary}
\newtheorem{claim}{Claim}
\newtheorem*{claim*}{Claim}

\newenvironment{claimproof}[1]{\par\noindent\textit{Proof of the claim:}\space#1}{\hfill $\blacksquare$\vspace{0.2cm}}

\newenvironment{claimproof*}[1]{\par\noindent\textit{Proof of the claim:}\space#1}{}

\newtheorem{maintheorem}{Theorem}
\newtheorem{maincorollary}[maintheorem]{Corollary}
\theoremstyle{definition}

\newtheorem{remark}[theorem]{Remark}
\newtheorem*{definition*}{Definition}
\newtheorem*{remark*}{Remark}

\newcommand{\suchthat}{\;\ifnum\currentgrouptype=16 \middle\fi|\;}

\DeclareMathOperator{\Aut}{\mathrm{Aut}}
\DeclareMathOperator{\PSL}{\mathrm{PSL}}

\DeclareMathOperator{\Sym}{\mathrm{Sym}}

\DeclareMathOperator{\Z}{\mathbf{Z}}

\newcommand{\HH}[1]{{}^{#1}\!H}
\newcommand{\hh}[1]{{}^{#1}\!h}
\newcommand{\PP}[1]{{}^{#1}\!\mathcal{P}}
\newcommand{\LL}[1]{{}^{#1}\!\mathcal{L}}
\newcommand{\II}[1]{{}^{#1}\!I}
\newcommand{\pipi}[1]{{}^{#1}\!\pi}
\newcommand{\PsiPsi}[1]{{}^{#1}\!\Psi}

\title{A homogeneous $\tilde{A}_2$-building with a non-discrete\\ automorphism group is Bruhat--Tits}
\author{Nicolas Radu\thanks{F.R.S.-FNRS Research Fellow.}}

\affil{UCLouvain, 1348 Louvain-la-Neuve, Belgium}
\date{March 15, 2018}

\begin{document}

% TITLE

\maketitle

% ABSTRACT

\begin{abstract}
Let $\Delta$ be a locally finite thick building of type $\tilde{A}_2$. We show that, if the type-preserving automorphism group $\Aut(\Delta)^+$ of $\Delta$ is transitive on panels of each type, then either $\Delta$ is Bruhat--Tits or $\Aut(\Delta)$ is discrete. For $\tilde{A}_2$-buildings which are not panel-transitive but only vertex-transitive, we give additional conditions under which the same conclusion holds. We also find a local condition under which an $\tilde{A}_2$-building is ensured to be exotic (i.e.\ not Bruhat--Tits). It can be used to show that the number of exotic $\tilde{A}_2$-buildings with thickness $q+1$ and admitting a panel-regular lattice grows super-exponentially with $q$ (ranging over prime powers). All those exotic $\tilde{A}_2$-buildings have a discrete automorphism group.
\end{abstract}

% TABLE OF CONTENTS

\tableofcontents

\section{Introduction}

A (locally finite) thick $\tilde{A}_2$-building $\Delta$ can be characterized as a simply connected simplicial complex of dimension~$2$ such that all simplicial spheres of radius~$1$ around vertices are isomorphic to the incidence graph of a (finite) projective plane. In this paper, $\Delta$ will always be a locally finite thick $\tilde{A}_2$-building and we will see $\Delta$ as a simplicial complex. The simplices of dimension~$2$ in $\Delta$ (i.e.\ triangles) are the \textbf{chambers} of $\Delta$, and those of dimension~$1$ (i.e.\ edges) are the \textbf{panels} of $\Delta$. Of course a \textbf{vertex} of $\Delta$ is a simplex of dimension~$0$. To each vertex of $\Delta$ we associate a \textbf{type} in $\{0,1,2\}$, so that all chambers have a vertex of each type. We also define the \textbf{type} of a panel in $\Delta$ as $\{a,b\}$ where $a, b \in \{0,1,2\}$ are the types of the two vertices of the panel. Thus each chamber has one panel of each type ($\{0,1\}$, $\{1,2\}$ and $\{0,2\}$). (Note that what we call \textit{type} here is generally called \textit{cotype} in the literature.)

In this paper, $\Aut(\Delta)$ denotes the full automorphism group of $\Delta$ (as a simplicial complex), while $\Aut(\Delta)^+$ is the subgroup of $\Aut(\Delta)$ consisting of the automorphisms which preserve the types. It is clear that $[\Aut(\Delta) : \Aut(\Delta)^+] \leq 6$, so that the locally compact group $\Aut(\Delta)$ (equipped with the topology of pointwise convergence) is non-discrete if and only if $\Aut(\Delta)^+$ is non-discrete.

The known sources of examples of $\tilde{A}_2$-buildings are the following:
\begin{enumerate}[(1)]
\item Following \cite{WeissAffine}, we say that $\Delta$ is \textbf{Bruhat--Tits} if its spherical building at infinity (which is a compact projective plane) is \textit{Moufang} (see \S\ref{section:preliminaries} for the definition of a Moufang projective plane). The only (locally finite and thick) Bruhat--Tits $\tilde{A}_2$-buildings are the ones associated to $\mathrm{PGL}(3,D)$ for $D$ a finite dimensional division algebra over a local field, see~\cite{WeissAffine}*{Chapter~28}. In particular, it follows that $\Aut(\Delta)$ is non-discrete when $\Delta$ is Bruhat--Tits. An $\tilde{A}_2$-building which is not Bruhat--Tits is called \textbf{exotic}.
\item One can construct $\tilde{A}_2$-buildings inductively, starting from a point $O$ and gluing triangles to the ball $B(O,r)$ of radius $r$ to obtain $B(O,r+1)$. This kind of construction is explained in~\cite{Ronan} and~\cite{Barre-Pichot}, where it was observed that $\tilde{A}_2$-buildings can be ``really" exotic. It is however rather hard to have any information on the automorphism group of a building constructed in that way.
\item $\tilde{A}_2$-buildings with lattices have been studied a lot: some of them with a panel-regular lattice (see~\cite{Essert}, \cite{Witzel} and \S\ref{section:Singer} below), others with a vertex-regular lattice (see~\cite{CMSZ} and~\cite{CMSZ2}), and also some with a lattice having two orbits of vertices (see~\cite{Barre}*{\S3} and~\cite{CapraceSixteen}*{Remark~8}). For the small examples, i.e.\ the ones with a small enough \textbf{thickness} (the number of chambers adjacent to a single panel), it could be checked with a computer that the automorphism group was discrete as soon as the building was exotic. Note that there exist exotic $\tilde{A}_2$-buildings with lattices and with arbitrarily large thickness, see~\cite{BCL}*{Appendix~D} or \cite{CapraceSixteen}*{Remark~8}.
\item Some exotic $\tilde{A}_2$-buildings can also be constructed from valuations on planar ternary rings, see~\cite{HVM}. The automorphism group of the $\tilde{A}_2$-buildings constructed in that way in~\cite{HVM2}*{\S7} is vertex-transitive and non-discrete, but it fixes a vertex at infinity, and is thus not unimodular by \cite{Caprace-Monod2}*{Theorem~M} (in particular, it cannot contain any lattice). This remark about the non-unimodularity will be important when stating Theorem~\ref{maintheorem:B} below.
\end{enumerate}

The goal of this paper is to provide sufficient conditions under which an exotic $\tilde{A}_2$-building has a discrete automorphism group. Our main result is the following.

\begin{maintheorem}\label{maintheorem:A}
Let $\Delta$ be a locally finite thick $\tilde{A}_2$-building and suppose that $\Aut(\Delta)^+$ is transitive on panels of each type. Then either:
\begin{enumerate}[(a)]
\item $\Delta$ is Bruhat--Tits; or
\item $\Aut(\Delta)$ is discrete.
\end{enumerate}
\end{maintheorem}

In the text we actually state and prove Theorem~\ref{maintheorem:A'} which is a more precise version of Theorem~\ref{maintheorem:A}. The same will be true for our other main results: alternative statements can be found in the text.

A natural question to ask is whether the panel-transitivity can be weakened in this theorem, and for instance replaced by vertex-transitivity. Because of the $\tilde{A}_2$-building described in~\cite{HVM2}*{\S7} (see (4) above), such a result cannot be true in these general terms. The next theorem however gives additional hypotheses under which a similar conclusion can be obtained.

\begin{maintheorem}\label{maintheorem:B}
Let $\Delta$ be a locally finite thick $\tilde{A}_2$-building. Suppose that $\Aut(\Delta)$ is transitive on vertices and unimodular, that $\Aut(\Delta)^+$ is transitive on vertices of each type, and that $\Delta$ has thickness $p+1$ for some prime $p$. Then either:
\begin{enumerate}[(a)]
\item $\Delta$ is Bruhat--Tits; or
\item $\Aut(\Delta)$ is discrete.
\end{enumerate}
\end{maintheorem}

Theorem~\ref{maintheorem:B} can in particular be applied to the locally finite thick $\tilde{A}_2$-buildings $\Delta$ with a vertex-regular lattice (see~(3) above) as soon as the thickness of $\Delta$ is $p+1$ for some prime $p$ (i.e.\ the local projective planes in $\Delta$ have order $p$).

As pointed out by the referee, the question whether the automorphism group of an exotic $\tilde{A}_2$-building admitting a cocompact lattice is always discrete was asked by Tim Steger in talks given in Blaubeuren and Orléans in 2007. Theorems~\ref{maintheorem:A} and~\ref{maintheorem:B} provide partial answers to that question.

Our results can also be viewed as giving weak hypotheses on $\Aut(\Delta)$ under which $\Delta$ is automatically Bruhat--Tits. It was proved in~\cite{VMVS} by Van Maldeghem and Van Steen that $\Delta$ is Bruhat--Tits as soon as $\Aut(\Delta)$ is \textit{Weyl-transitive}. Recall that $\Aut(\Delta)$ is \textbf{Weyl-transitive} if for any two pairs of chambers $(c,d)$, $(c',d')$ in $\Delta$ with equal Weyl-distances ($\delta(c,d) = \delta(c',d')$), there exists $g \in \Aut(\Delta)^+$ such that $g(c) = c'$ and $g(d) = d'$. Theorem~\ref{maintheorem:A} actually shows that having $\Aut(\Delta)^+$ transitive on panels of each type and non-discrete (which is strictly weaker than requiring the Weyl-transitivity) is already sufficient to have the same conclusion. Our proof of Theorem~\ref{maintheorem:A} actually uses the machinery developed by the authors in~\cite{VMVS}. 

Note that the fact that Weyl-transitivity implies Bruhat--Tits was later proved to be true in any Euclidean building. Indeed, if $X$ is an Euclidean building and if $\Aut(X)$ is Weyl-transitive, then $\Aut(X^\infty)$ is \textbf{strongly transitive} on the building at infinity $X^\infty$ (i.e.\ transitive on pairs $(A,c)$ where $A$ is an apartment of $X^\infty$ and $c$ is a chamber of $A$), which implies that $\Aut(X)$ is strongly transitive on $X$ by \cite{CaCi}*{Theorem~1.1} and then that $X$ is Bruhat--Tits by~\cite{Caprace-Monod}*{Corollary~E} or~\cite{Radu}*{Corollary~B}.

\medskip

The following result is of different nature but is somewhat complementary to Theorems~\ref{maintheorem:A} and~\ref{maintheorem:B}. Indeed, it gives a local condition under which an $\tilde{A}_2$-building is ensured to be exotic.

\begin{maintheorem}\label{maintheorem:exotic}
Let $\Delta$ be a locally finite thick $\tilde{A}_2$-building, let $x_0$ and $x_1$ be two adjacent vertices in $\Delta$ and let $C$ be the set of chambers adjacent to both $x_0$ and $x_1$. For each $j \in \{0,1\}$, let $G_j \leq \Sym(C)$ be the image of $\Aut(\Pi_{x_j})(x_{1-j})$ in $\Sym(C)$, where $\Pi_{x_j}$ is the projective plane at $x_j$. If $G_0 \neq G_1$, then $\Delta$ is exotic.
\end{maintheorem}

We apologize that the condition is indeed that the groups $G_0$ and $G_1$ do not coincide as subgroups of $\Sym(C)$. In particular they might very well be isomorphic.

Our theorems can be applied in the context of \textit{Singer cyclic lattices}. Following \cite{Witzel}, a \textbf{Singer cyclic lattice} is a group $\Gamma \leq \Aut(\Delta)$ acting simply transitively (i.e.\ regularly) on the panels of each type of an $\tilde{A}_2$-building $\Delta$ and such that each vertex stabilizer in $\Gamma$ is cyclic. It is called \textbf{exotic} if $\Delta$ is exotic, and the \textbf{parameter} of $\Gamma$ is the order of the local projective planes in $\Delta$.

\begin{maincorollary}\label{maincorollary:bound}
For each $q \geq 2$, there are at most $\left(\frac{q(q^2-1)}{3}\right)^2$ isomorphism classes of non-exotic Singer cyclic lattices with parameter $q$.
\end{maincorollary}

In \cite{Witzel}, Witzel conjectured that almost all Singer cyclic lattices are exotic and pairwise not quasi-isometric. Using the fact that the total number of Singer cyclic lattices with parameter $q$ grows super-exponentially with $q$ (see \cite{Witzel}*{Theorem~B}), our previous result thus solves the first part of that conjecture.

\begin{maincorollary}\label{maincorollary:limit}
Almost all Singer cyclic lattices are exotic in the following sense:
$$\lim_{q \to \infty} \frac{|\{\text{exotic Singer cyclic lattices with parameter $q$}\}/\sim\!|}{|\{\text{Singer cyclic lattices with parameter $q$}\}/\sim\!|} = 1,$$
where $q$ ranges over prime powers and $\sim$ denotes the isomorphism relation.
\end{maincorollary}

It is also a consequence of our discreteness result that all exotic Singer cyclic lattices live in an $\tilde{A}_2$-building with a discrete automorphism group. Using the fact that cocompact lattices in $\tilde{A}_2$-buildings are QI-rigid \cite{Kleiner-Leeb}, this in particular implies that they have finite index in their abstract commensurator group. This can be seen as an analog of the result of Margulis stating that an irreducible lattice in a connected semisimple Lie group $G$ with finite center and no compact factors is arithmetic as soon as it has infinite index in its commensurator in $G$, see \cite{Margulis}*{Theorem~IX.1.16}.

\subsection*{Acknowledgement}

I am grateful to Pierre-Emmanuel Caprace for suggesting me this problem and for encouraging me to read the inspiring work~\cite{VMVS} of Van Maldeghem and Van Steen.

\section{Preliminaries about Hjelmslev planes}\label{section:preliminaries}

This section gives the definition and first properties of Hjelmslev planes, which will be of central importance in the whole text. It is largely inspired from~\cite{VMVS}.

Given a vertex $O$ in $\Delta$ and a natural number $n \geq 1$, we define the geometry $\HH n(O)$ as follows. The geometry $\HH 1(O)$ is just the residue of $O$, which is a projective plane. So the points of $\HH 1(O)$ are certain vertices of $\Delta$ adjacent to $O$, and similarly for the lines of $\HH 1(O)$. Now for $n \geq 1$, the \textbf{points} (resp.\ \textbf{lines}) of $\HH n(O)$ are the sequences $(v_1, v_2, \ldots, v_n)$ of vertices of $\Delta$, where $v_1$ is a point (resp.\ a line) of $\HH 1(O)$ and $\{v_{i-1}, v_{i+1}\}$ is a pair of non-incident point and line in $\HH 1(v_i)$ (where $v_0 := O$). We will sometimes identify an element $(v_1, \ldots, v_n)$ of $\HH n(O)$ with the vertex $v_n$ of $\Delta$ (the other vertices $v_1, \ldots, v_{n-1}$ being uniquely determined by $v_n$). A point $(p_1, p_2, \ldots, p_n)$ of $\HH n(O)$ is \textbf{incident} with a line $(\ell_1, \ell_2, \ldots, \ell_n)$ if all vertices $O, p_1, \ldots, p_n, \ell_1, \ldots, \ell_n$ are contained in a common apartment and if $p_1$ and $\ell_1$ are adjacent in $\Delta$. This geometry $\HH n(O)$ is called a \textbf{projective Hjelmslev plane of level $n$}. When the vertex $O$ has no real importance, we write $\HH n$ instead of $\HH n(O)$. The point set (resp.\ line set) of $\HH n$ is then $\PP n$ (resp.\ $\LL n$), while incidence is denoted by $\II n$.

For $i \leq n$, the natural morphism from $\HH n$ to $\HH i$ is denoted by $\pipi i$. If $P, Q \in \PP n$ satisfy $\pipi i(P) = \pipi i(Q)$ for some $0 < i \leq n$, then we call $P$ and $Q$ \textbf{$i$-neighboring}. For $i = 1$ we just say that $P$ and $Q$ are \textbf{neighboring}. Similarly for lines. Also, if $P \in \PP n$ and $\ell \in \LL n$ are such that $\pipi i(P)\ \II i\ \pipi i(L)$ for some $0 < i \leq n$ then we say that $P$ is \textbf{$i$-near} $\ell$. Once again, $P$ is \textbf{near} $L$ when $i = 1$.

A \textbf{collineation} $\alpha$ of $\HH n$ is, as usual, a bijection from $\PP n$ to itself and a bijection from $\LL n$ to itself which preserve $\II n$. It is not hard to see that all $i$-neighboring relations are determined by the geometry of $\HH n$, so that every collineation $\alpha$ of $\HH n$ induces a unique collineation $\alpha^{\star_i}$ of $\HH i$. When acting on elements of $\HH i$, $\alpha^{\star_i}$ will sometimes be replaced by $\alpha$, so as to simplify the notation. For a fixed vertex $O$ in $\Delta$, the group of all collineations of $\HH n(O)$ which are induced from an automorphism in $\Aut(\Delta)^+$ fixing $O$ will be denoted by $\PsiPsi n(O)$. When $\alpha \in \PsiPsi n(O)$ is induced by $g \in \Aut(\Delta)^+$, it will be convenient to talk about the action of $\alpha$ (instead of $g$) on $\Delta$.

Given $P \in \PP n$ and $\ell \in \LL n$ with $P \ \II n \ \ell$, an \textbf{elation} of $\HH n$ with axis $\ell$ and center $P$ is a collineation of $\HH n$ fixing all points incident with $\ell$ and fixing all lines incident with $P$. As the next lemma shows, such an elation also fixes additional points and lines.

\begin{lemma}\label{lemma:quasi-elation}
Let $\alpha$ be an elation of $\HH n$ ($n \geq 2$) with axis $\ell$ and center $P$. Then $\alpha$ fixes all points (resp.\ lines) of $\HH n$ that are $(n-1)$-neighboring $P$ (resp.\ $\ell$).
\end{lemma}

\begin{proof}
See~\cite{VMVS}*{Lemma~5}.
\end{proof}

An elation $\alpha$ of $\HH n$ such that $\alpha^{\star_{n-1}}$ is trivial is called a \textbf{$\hh 1$-collineation} of $\HH n$. (All elations of $\HH 1$ are $\hh 1$-collineations.) By definition, an elation $\alpha$ with axis $\ell$ and center $P$ fixes all points incident with $\ell$ and all lines incident with $P$. The following lemma states that when $\alpha$ is a $\hh 1$-collineation, it also fixes the points near $\ell$ and the lines near $P$.

\begin{lemma}\label{lemma:h1}
Let $\alpha$ be a $\hh 1$-collineation of $\HH n$ with axis $\ell$ and center $P$. Then $\alpha$ fixes all points (resp.\ lines) of $\HH n$ that are near $\ell$ (resp.\ $P$).
\end{lemma}

\begin{proof}
See~\cite{VMVS}*{Lemma~14}.
\end{proof}

We then get the following result as a direct consequence.

\begin{lemma}\label{lemma:changeaxis}
Let $\alpha$ be a $\hh 1$-collineation of $\HH n$ with axis $\ell$ and center $P$. Then for each $\ell' \in \LL n$ neighboring $\ell$ and each $P' \in \PP n$ neighboring $P$, $\alpha$ is also a $\hh 1$-collineation with axis $\ell'$ and center $P'$.
\end{lemma}

\begin{proof}
By Lemma~\ref{lemma:h1}, $\alpha$ fixes all points (resp.\ lines) of $\HH n$ that are near $\ell$ (resp.\ $P$). Given $P'$ neighboring $P$ and $\ell'$ neighboring $\ell$, this is equivalent to saying that $\alpha$ fixes all points (resp.\ lines) that are near $\ell'$ (resp.\ $P'$). In particular, $\alpha$ fixes all points (resp.\ lines) incident with $\ell'$ (resp.\ $P'$), which means that $\alpha$ is an elation (and thus a $\hh 1$-collineation) with axis $\ell'$ and center $P'$.
\end{proof}

The following lemma also comes from \cite{VMVS}. For $n = 1$, this is a particular case of a well-known result of Tits \cite{Tits}*{Theorem~4.1.1}.

\begin{lemma}\label{lemma:411easy}
Let $\alpha$ be a non-trivial $\hh 1$-collineation of $\HH n$ with axis $\ell$ and center $P$. Then $\alpha$ does not fix any point (resp.\ line) of $\HH n$ which is not near $\ell$ (resp.\ $P$).
\end{lemma}

\begin{proof}
See \cite{VMVS}*{Lemma~16~(iv)}.
\end{proof}

From this lemma we can easily deduce the more general next result.

\begin{lemma}\label{lemma:free}
Let $\alpha$ be a non-trivial elation of $\HH n$ with axis $\ell$ and center $P$. Then $\alpha$ does not fix any point (resp.\ line) of $\HH n$ which is not near $\ell$ (resp.\ $P$). In particular, if $m \in \LL n$ is incident with $P$ but not neighboring $\ell$, then the group of all elations with axis $\ell$ and center $P$ acts freely on the points incident with $m$ but not neighboring $P$.
\end{lemma}

\begin{proof}
Let us prove it by induction on $n$. For $n = 1$, this is equivalent to Lemma~\ref{lemma:411easy}. Now assume the assertion is proved in $\HH{n-1}$ and let $\alpha$ be a non-trivial elation of $\HH n$ with axis $\ell$ and center $P$. It is clear that $\alpha^{\star_{n-1}}$ is an elation of $\HH{n-1}$, with axis $\pipi{n-1}(\ell)$ and center $\pipi{n-1}(P)$. If $\alpha^{\star_{n-1}}$ is trivial then $\alpha$ is a $\hh 1$-collineation of $\HH n$ and we can directly apply Lemma~\ref{lemma:411easy} to conclude. If on the contrary $\alpha^{\star_{n-1}}$ is not trivial then it is a non-trivial elation of $\HH{n-1}$ and the result follows from the induction hypothesis.
\end{proof}

We now explain what it means for $\HH n$ to be \textit{Moufang}. First fix $P \in \PP n$ and $\ell \in \LL n$ with $P \ \II n \ \ell$. Given $m \in \LL n$ incident with $P$ but not neighboring $\ell$, we say that $\HH n$ is \textbf{$(P, \ell)$-transitive} if the group of all elations with axis $\ell$ and center $P$ acts transitively on the points incident with $m$ but not neighboring $P$. In view of Lemma~\ref{lemma:free}, this condition does not depend on the choice for $m$ and the action is then automatically simply transitive. When $\HH n$ is $(P, \ell)$-transitive for all $P \in \PP n$ and all $\ell \in \LL n$ with $P \ \II n \ \ell$, we say that $\HH n$ is \textbf{Moufang}. For $n = 1$, this definition is of course equivalent to the definition of a Moufang projective plane.

Given $n \geq 1$ we say that $\Delta$ is \textbf{$n$-Moufang} if $\HH n(O)$ is Moufang for each vertex $O$ in $\Delta$. Being $n$-Moufang is clearly weaker than being $(n+1)$-Moufang. As the next lemma shows, if $\Delta$ is $n$-Moufang for each $n \geq 1$ then the projective plane $\Delta^\infty$ at infinity of $\Delta$ is Moufang, i.e.\ $\Delta$ is Bruhat--Tits. The proof of this lemma essentially comes from~\cite{VMVS2}*{\S5}.

\begin{lemma}\label{lemma:nMoufang}
Suppose that $\Delta$ is $n$-Moufang for each $n \geq 1$. Then the projective plane $\Delta^{\infty}$ is Moufang, i.e.\ $\Delta$ is Bruhat--Tits.
\end{lemma}

\begin{proof}
Consider $\ell^\infty$ and $m^\infty$ two lines in $\Delta^\infty$, and denote by $P^\infty$ the point of $\Delta^\infty$ incident to $\ell^\infty$ and $m^\infty$. We want to show that $\Delta^\infty$ is $(P^\infty,\ell^\infty)$-transitive, i.e.\ that the group of all elations of $\Delta^\infty$ with axis $\ell^\infty$ and center $P^\infty$ acts transitively on the points incident with $m^\infty$ but different from $P^\infty$. Consider $Q^\infty$ and $R^\infty$ two points incident with $m^\infty$, different from $P^\infty$. Let $A^\infty$ be some apartment in $\Delta^\infty$ containing $\ell^\infty$, $P^\infty$, $m^\infty$ and $Q^\infty$. There exists an apartement $A$ in $\Delta$ whose apartement at infinity is $A^\infty$. Now choose a vertex $O$ in $A$ so that the two open rays from $O$ to $P^\infty$ and $R^\infty$ are disjoint. For each $n \geq 1$, $\HH n(O)$ is Moufang, so there exists an elation $\alpha_n$ of $\HH n(O)$ with axis $\pipi n(\ell^\infty)$ and center $\pipi n(P^\infty)$, sending $\pipi n(Q^\infty)$ to $\pipi n(R^\infty)$ (where $\pipi n(x^\infty)$ is the point or line of $\HH n(O)$ represented by the ray from $O$ to $x^\infty$). Lemma~\ref{lemma:free} implies that $\alpha_n^{\star_{k}} = \alpha_k$ for each $1 \leq k \leq n$. We can thus consider the inverse limit $\alpha$ of the sequence $(\alpha_n)$, which is an elation of $\Delta^\infty$ with axis $\ell^\infty$ and center $P^\infty$, sending $Q^\infty$ to $R^\infty$.
\end{proof}

Finally, for our future needs we give a name to some vertices of $\Delta$. Given $P \in \PP n(O)$ and $\ell \in \LL n(O)$ with $P \ \II n \ \ell$ (where $O$ is a vertex of $\Delta$), the consecutive vertices of the geodesic from $P$ to $\ell$ in $\Delta$ are denoted by $P = v_0(P, \ell), v_1(P, \ell), \ldots, v_n(P, \ell) = \ell$.

\section{Panel-transitive \texorpdfstring{$\tilde{A}_2$-buildings}{triangle buildings}}

Given $n \geq 1$, we say that $\Aut(\Delta)$ (or $\Aut(\Delta)^+$) is \textbf{$n$-discrete} if there exists a vertex $O$ in $\Delta$ such that the only element of $\Aut(\Delta)$ fixing $O$ and acting trivially on $\HH n(O)$ is the identity. Being $n$-discrete is clearly stronger than being $(n+1)$-discrete. Then $\Aut(\Delta)$ (or $\Aut(\Delta)^+$) is \textbf{non-$n$-discrete} if for each vertex $O$ in $\Delta$ there exists a non-trivial element of $\Aut(\Delta)$ fixing $O$ and acting trivially on $\HH n(O)$. Remark that $\Aut(\Delta)$ is non-discrete if and only if it is non-$n$-discrete for all $n \geq 1$. In this section we prove Theorem~\ref{maintheorem:A'} which is thus a more precise version of Theorem~\ref{maintheorem:A} (in view of Lemma~\ref{lemma:nMoufang}).

\begin{primetheorem}{maintheorem:A}\label{maintheorem:A'}
Let $\Delta$ be a locally finite thick $\tilde{A}_2$-building and suppose that $\Aut(\Delta)^+$ is transitive on panels of each type. Then for each $n \geq 1$, at least one of the following assertions holds:
\begin{enumerate}[(a)]
\item $\Delta$ is $n$-Moufang, or
\item $\Aut(\Delta)$ is $(6n+2)$-discrete.
\end{enumerate}
\end{primetheorem}

In the proof, we assume that $\Aut(\Delta)^+$ is transitive on panels of each type and non-$(6n+2)$-discrete, and aim to show that $\HH n(O)$ is Moufang for each vertex $O$ in $\Delta$.

\subsection{Constructing \texorpdfstring{$\hh 1$-collineations}{h1-collineations}}

In this first subsection, we observe that the non-$(n+3)$-discreteness of $\Aut(\Delta)$ together with its transitivity on vertices of each type implies the existence of non-trivial $\hh 1$-collineations in $\PsiPsi n(O)$ for each vertex $O$ in $\Delta$.
We start with an easy result valid in any $\tilde{A}_2$-building $\Delta$.

\begin{lemma}\label{lemma:segment}
Let $v_0, \ldots, v_k$ ($k \geq 1$) be consecutive vertices of a wall of $\Delta$. Consider vertices $w_0, \ldots, w_{k-1}$ with $w_i$ adjacent to $v_i$, $v_{i+1}$ and $w_{i-1}$ (if $i \geq 1$) for each $i \in \{0, \ldots, k-1\}$. Similarly, consider vertices $w'_0, \ldots, w'_{k-1}$ with $w'_i$ adjacent to $v_i$, $v_{i+1}$ and $w'_{i-1}$ (if $i \geq 1$) for each $i \in \{0, \ldots, k-1\}$. If $g \in \Aut(\Delta)^+$ fixes $v_0, \ldots, v_k$ and if $g(w_0) = w'_0$, then $g(w_i) = w'_i$ for each $i \in \{0, \ldots, k-1\}$.
\end{lemma}

\begin{proof}
For each $i \in \{1, \ldots, k-1\}$, the fact that $g$ fixes $v_{i-1}$, $v_i$ and $v_{i+1}$ clearly implies that $g$ sends $w_{i-1}$ to $w'_{i-1}$ if and only if $g$ sends $w_i$ to $w'_i$ (see Figure~\ref{picture:segment} for an illustration). The conclusion then follows immediately.
\begin{figure}[h!]
\centering
\begin{pspicture*}(-2.9,-1.5)(4.00,1.4)
%\fontsize{11pt}{11pt}\selectfont
\psset{unit=1.2cm}

\psline(-2,0)(3,0)
\psline(-1.5,0.866)(2.5,0.866)
\psline(-1.5,-0.866)(2.5,-0.866)

% Bande sup
\psline(-2,0)(-1.5,0.866)
\psline(-1.5,0.866)(-1,0)
\psline(-1,0)(-0.5,0.866)
\psline(-0.5,0.866)(0,0)
\psline(0,0)(0.5,0.866)
\psline(0.5,0.866)(1,0)
\psline(1,0)(1.5,0.866)
\psline(1.5,0.866)(2,0)
\psline(2,0)(2.5,0.866)
\psline(2.5,0.866)(3,0)

% Bande inf
\psline(-2,0)(-1.5,-0.866)
\psline(-1.5,-0.866)(-1,0)
\psline(-1,0)(-0.5,-0.866)
\psline(-0.5,-0.866)(0,0)
\psline(0,0)(0.5,-0.866)
\psline(0.5,-0.866)(1,0)
\psline(1,0)(1.5,-0.866)
\psline(1.5,-0.866)(2,0)
\psline(2,0)(2.5,-0.866)
\psline(2.5,-0.866)(3,0)

% Notation
\psdot[linewidth=0.03](-2,0)
\rput(-2.22,0.115){$v_0$}
\psdot[linewidth=0.03](-1,0)
\rput(-1.28,0.12){$v_1$}
\psdot[linewidth=0.03](2,0)
\rput(2.47,0.115){$v_{k-1}$}
\psdot[linewidth=0.03](3,0)
\rput(3.2,0.1){$v_k$}
\psdot[linewidth=0.03](-1.5,0.866)
\rput(-1.55,1.05){$w_0$}
\psdot[linewidth=0.03](2.5,0.866)
\rput(2.55,1.05){$w_{k-1}$}
\psdot[linewidth=0.03](-1.5,-0.866)
\rput(-1.53,-1.08){$w'_0$}
\psdot[linewidth=0.03](2.5,-0.866)
\rput(2.68,-1.08){$w'_{k-1}$}
\end{pspicture*}
\caption{Illustration of Lemma~\ref{lemma:segment}.}\label{picture:segment}
\end{figure}
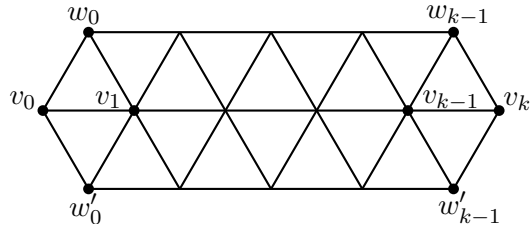
\end{proof}

This enables us to show the following.

\begin{lemma}\label{lemma:triviality}
Let $O$ be a vertex of $\Delta$ and let $\alpha \in \PsiPsi n(O)$ ($n \geq 2$) be a non-trivial collineation such that $\alpha^{\star_{n-1}}$ is trivial. Then there exists $P \in \PP n(O)$ and $\ell \in \LL n(O)$ with $P\ \II n\ \ell$ and such that $\alpha$ does not fix $v_1(P, \ell), v_2(P, \ell), \ldots, v_{n-1}(P, \ell)$.
\end{lemma}

\begin{proof}
For any $P \in \PP n(O)$ and $\ell \in \LL n(O)$ with $P\ \II n\ \ell$ we know by Lemma~\ref{lemma:segment} (and since $\alpha^{\star_{n-1}}$ is trivial) that either all vertices $v_1(P, \ell), \ldots, v_{n-1}(P, \ell)$ are fixed by $\alpha$ or none of them is fixed by $\alpha$.

We therefore proceed by contradiction, assuming that for all such $P$ and $\ell$, all the vertices $v_1(P, \ell), \ldots, v_{n-1}(P, \ell)$ are fixed by $\alpha$. We show that, in this case, $\alpha$ is trivial (which gives the contradiction).

Consider any point $P \in \PP n(O)$ and choose two lines $\ell, \ell' \in \LL n(O)$ incident with $P$ and such that $\ell$ and $\ell'$ are not neighboring. Then $\alpha$ fixes $v_1(P, \ell)$, $v_1(P,\ell')$ and $\pipi{n-1}(P)$, so it must fix $P$. This can be done for any choice of a point $P \in \PP n(O)$, and similarly for any choice of a line $\ell \in \LL n(O)$, so $\alpha$ is trivial.
\end{proof}

\begin{proposition}\label{proposition:h1col}
Let $n \geq 1$ and suppose that $\Aut(\Delta)$ is non-$(n+3)$-discrete and transitive on vertices of each type. Then for each vertex $O$ in $\Delta$, there exists a non-trivial $\hh 1$-collineation in $\PsiPsi n(O)$.
\end{proposition}

\begin{proof}
In view of the transitivity of $\Aut(\Delta)$ on vertices of each type, it suffices to find three vertices $O_0$, $O_1$, $O_2$ of types $0$, $1$ and $2$ such that $\PsiPsi n(O_i)$ contains a non-trivial $\hh 1$-collineation for each $i \in \{0,1,2\}$.

Fix some vertex $O$ in $\Delta$. The non-$(n+3)$-discreteness of $\Aut(\Delta)$ implies that there exists $N \geq n+4$ such that $\PsiPsi N(O)$ contains a non-trivial collineation $\alpha$ with $\alpha^{\star_{N-1}}$ trivial. By Lemma~\ref{lemma:triviality}, there exists $P \in \PP N(O)$ and $\ell \in \LL N(O)$ with $P\ \II N \ \ell$ and such that none of the vertices $v_1(P,\ell), \ldots, v_{N-1}(P, \ell)$ is fixed by $\alpha$. Now write $X = \pipi{N-n}(P)$ and $Y = \pipi{N-n}(\ell)$ (see Figure~\ref{picture:h1col}). As $N-n \geq 4$, the geodesic from $X$ to $Y$ in $\Delta$ contains at least three vertices different from $X$ and $Y$. Since three consecutive vertices in such a configuration always have the three different types, there exist $O_0$, $O_1$ and $O_2$ with types $0$, $1$ and $2$ and strictly between $X$ and $Y$. For each $i \in \{0,1,2\}$, the action induced by $\alpha$ on $\HH n(O_i)$ is non-trivial, because $\alpha$ acts non-trivially on $v_1(P,\ell), \ldots, v_{N-1}(P,\ell)$. There remains to check that it is a $\hh 1$-collineation of $\HH n(O_i)$, but this is a consequence of the fact that $\alpha^{\star_{N-1}}$ is trivial.
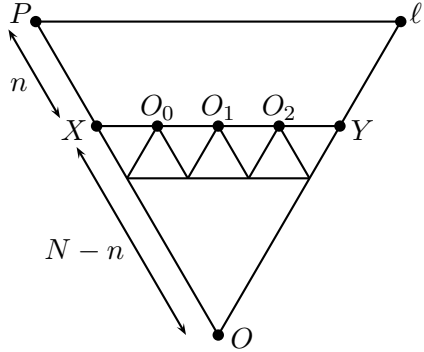
\begin{figure}[h!]
\centering
\begin{pspicture*}(-4,-0.2)(4,4.5)
%\fontsize{11pt}{11pt}\selectfont
\psset{unit=0.8cm}

% Hexagone
\psline(0,0)(-3,5.196)
\psline(0,0)(3,5.196)
\psline(-3,5.196)(3,5.196)
\psline(-2,3.464)(2,3.464)

% Chambres
\psline(-1.5,2.598)(-1,3.464)
\psline(-0.5,2.598)(-1,3.464)
\psline(-0.5,2.598)(0,3.464)
\psline(0.5,2.598)(0,3.464)
\psline(0.5,2.598)(1,3.464)
\psline(1.5,2.598)(1,3.464)
\psline(-1.5,2.598)(1.5,2.598)

% Notation
\psdot[linewidth=0.05](0,0)
\rput(0.39,-0.05){$O$}
\psdot[linewidth=0.05](-2,3.464)
\rput(-2.37,3.4){$X$}
\psdot[linewidth=0.05](-1,3.464)
\rput(-1,3.77){$O_0$}
\psdot[linewidth=0.05](0,3.464)
\rput(0,3.77){$O_1$}
\psdot[linewidth=0.05](1,3.464)
\rput(1,3.77){$O_2$}
\psdot[linewidth=0.05](2,3.464)
\rput(2.37,3.4){$Y$}
\psdot[linewidth=0.05](-3,5.196)
\rput(-3.25,5.34){$P$}
\psdot[linewidth=0.05](3,5.196)
\rput(3.25,5.34){$\ell$}

\psline[arrows=<->](-0.53,0)(-2.33,3.118)
\rput(-2.2,1.4){$N-n$}
\psline[arrows=<->](-2.6,3.594)(-3.45,5.067)
\rput(-3.28,4.15){$n$}
\end{pspicture*}
\caption{Illustration of Proposition~\ref{proposition:h1col}.}\label{picture:h1col}
\end{figure}
\end{proof}

The previous proposition shows the existence of a non-trivial $\hh 1$-collineation in $\PsiPsi n(O)$, in some circumstances. We already know some properties of such collineations (see Lemma~\ref{lemma:411easy}), but the next lemma is more precise.

\begin{lemma}\label{lemma:12}
Let $O$ be a vertex of $\Delta$ and consider $P \in \PP n(O)$ and $\ell \in \LL n(O)$ with $P \ \II n \ \ell$ ($n \geq 2$). Let also $Q \in \PP n(O)$ be a point not near $\ell$ and $o \in \LL n(O)$ be a line not near $P$, such that $Q \ \II n \ o$.
\begin{enumerate}[(i)]
\item Let $\alpha \in \PsiPsi n(O)$ be a non-trivial $\hh 1$-collineation with axis $\ell$ and center $P$. Then $\alpha$ does not fix $v_i(Q, o)$, for any $i \in \{0, 1, \ldots, n\}$.
\item Denote by $m \in \LL n(O)$ the line incident with $P$ and $Q$. Suppose that the group $G$ of all $\hh 1$-collineations in $\PsiPsi n(O)$ with axis $\ell$ and center $P$ acts transitively on the set of points $(n-1)$-neighboring $Q$ and incident with $m$. Then, for each $i \in \{0, \ldots, n-2\}$, $G$ acts transitively on the set of chambers of $\Delta$ having vertices $v_i(\pipi{n-1}(Q), \pipi{n-1}(o))$ and $v_{i+1}(\pipi{n-1}(Q), \pipi{n-1}(o))$ but not $v_i(\pipi{n-2}(Q), \pipi{n-2}(o))$.
\end{enumerate}
\end{lemma}

\begin{proof}
Let $\alpha \in \PsiPsi n(O)$ be a $\hh 1$-collineation with axis $\ell$ and center $P$. Let also $m \in \LL n(O)$ be the line incident with $P$ and $Q$ (see Figure~\ref{picture:12}). We know by definition of an elation that $\alpha$ fixes $m$, and the fact that $\alpha^{\star_{n-1}}$ is trivial implies that it also fixes $\pipi{n-1}(Q)$. Hence, from Lemma~\ref{lemma:segment} applied to the segment from $\pipi{n-1}(m)$ to $\pipi{n-1}(Q)$, we get that $\alpha$ fixes $v_1(Q, m)$. Assertions (i) and (ii) then follow thanks to another application of Lemma~\ref{lemma:segment} to the segment with vertices $v_1(Q, m), \pipi{n-1}(Q), v_1(\pipi{n-1}(Q), \pipi{n-1}(o)), \ldots, \pipi{n-1}(o)$. (Recall, for (i), that when $\alpha$ is non-trivial it does not fix $Q$ nor $o$ by Lemma~\ref{lemma:411easy}.)
\begin{figure}[t!]
\centering
\begin{pspicture*}(-3.7,-3)(3.7,3)
%\fontsize{11pt}{11pt}\selectfont
\psset{unit=0.8cm}

% Hexagone
\psline(0,0)(-4,0)
\psline(0,0)(2,3.464)
\psline(0,0)(-2,3.464)
\psline(0,0)(2,-3.464)
\psline(0,0)(-2,-3.464)

\psline(-4,0)(-2,3.464)
\psline(-2.5,2.598)(1.5,2.598)

% Chambre à gauche
\psline(-1.5,2.598)(-1.6,3.364)
\psline(-2.5,2.598)(-1.6,3.364)

% Chambres au milieu
\psline(0.5,2.598)(0,3.464)
\psline(-0.5,2.598)(0,3.464)
\psline(0.5,2.598)(0.4,3.364)
\psline(-0.5,2.598)(0.4,3.364)

% Notation
\psdot[linewidth=0.05](0,0)
\rput(0.44,0){$O$}
\psdot[linewidth=0.05](-2,3.464)
\rput(-2.37,3.48){$Q$}
\psdot[linewidth=0.05](-4,0)
\rput(-4.33,0.13){$m$}
\psdot[linewidth=0.05](2,3.464)
\rput(2.27,3.48){$o$}
\psdot[linewidth=0.05](-2,-3.464)
\rput(-2.34,-3.53){$P$}
\psdot[linewidth=0.05](2,-3.464)
\rput(2.23,-3.53){$\ell$}
\psdot[linewidth=0.05](-2.5,2.598)
\rput(-3.6,2.6){$v_1(Q,m)$}
%\rput(-0.32,2.23){$v_i(Q,o)$}
\end{pspicture*}
\caption{Illustration of Lemma~\ref{lemma:12}.}\label{picture:12}
\end{figure}
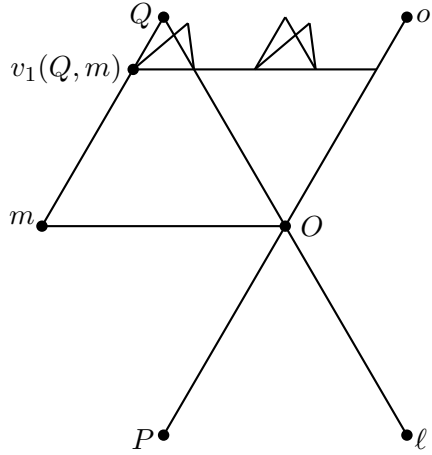
\end{proof}

\subsection{From panel-transitivity to chamber-transitivity}

In this subsection, we prove that if $\Aut(\Delta)^+$ is non-$4$-discrete and transitive on panels of each type, then $\Aut(\Delta)^+$ is transitive on chambers. We start by the following easy lemma, valid in any projective Hjelmslev plane of level $1$ (i.e.\ any projective plane).

\begin{lemma}\label{lemma:cycles}
Let $\alpha$ be a non-trivial elation of $\HH 1$ with axis $\ell$ and center $P$. Let $m \in \LL 1$ be incident with $P$ but different from $\ell$. Then the permutation induced by $\alpha$ on the set of $q$ points incident with $m$ but different from $P$ is a product of $k \geq 1$ disjoint cycles of the same length $c \geq 2$, where $k \cdot c = q$. Moreover, $k$ and $c$ do not depend on $m$.
\end{lemma}

\begin{proof}
Let $\sigma$ be the permutation induced by $\alpha$ on this set of $q$ points. By Lemma~\ref{lemma:411easy}, $\sigma$ has no fixed point. Now it suffices to prove that two cycles in the cycle decomposition of $\sigma$ always have the same length. Suppose for a contradiction that there are two cycles of different lengths $c_1 < c_2$. Then $\alpha^{c_1}$ is an elation of $\HH 1$ which is non-trivial (because $\sigma^{c_1}$ is non-trivial) and $\sigma^{c_1}$ has fixed points, which contradicts Lemma~\ref{lemma:411easy}. So all $k$ cycles in the cycle decomposition must have the same length $c \geq 2$, and of course $k \cdot c = q$. Note that $k$ and $c$ do not depend on $m$, otherwise we would once again get a power of $\alpha$ which is non-trivial but has forbidden fixed points.
\end{proof}

\begin{proposition}\label{proposition:panel}
Suppose that $\Aut(\Delta)^+$ is non-$4$-discrete and transitive on panels of each type. Then $\Aut(\Delta)^+$ is chamber-transitive.
\end{proposition}

\begin{proof}
Assume for a contradiction that $\Aut(\Delta)^+$ is not chamber-transitive. Then we can color the chambers of $\Delta$ in blue and red so that each color is used at least once and two chambers with different colors do not belong to the same orbit. (For instance, color one orbit of chambers in blue and all other orbits in red.) For each type $t \in \{\{0,1\},\{1,2\},\{0,2\}\}$, the transitivity on $t$-panels implies that there exist $b_t$ and $r_t$ such that all $t$-panels are adjacent to $b_t$ blue chambers and $r_t$ red chambers. Note that $b_t \geq 1$ and $r_t \geq 1$, otherwise all chambers of $\Delta$ would be the same color. In $\Delta$, all panels have the same number of chambers, say $q+1$, so $b_t + r_t = q+1$ for each $t$.

We first claim that $b_{\{0,1\}} = b_{\{1,2\}} = b_{\{0,2\}}$ (and $r_{\{0,1\}} = r_{\{1,2\}} = r_{\{0,2\}}$). Indeed, take $t, t' \in \{\{0,1\},\{1,2\},\{0,2\}\}$ with $t \neq t'$ and consider a vertex $v$ of type $t \cap t'$ in $\Delta$. The number of blue chambers adjacent to $v$ (i.e.\ in the residue corresponding to $v$) is equal to $p_t \cdot b_t$, where $p_t$ is the number of $t$-panels adjacent to $v$. Since the residue associated to $v$ is a projective plane of order $q$, we have $p_t = q^2+q+1$ and the number of blue chambers adjacent to $v$ is $(q^2+q+1) \cdot b_t$. But for the same reason with $t'$ instead of $t$, this number is also equal to $(q^2+q+1) \cdot b_{t'}$. So $b_t = b_{t'}$ and $r_t = r_{t'}$. In the following we therefore write $b = b_{\{0,1\}} = b_{\{1,2\}} = b_{\{0,2\}}$ and $r = r_{\{0,1\}} = r_{\{1,2\}} = r_{\{0,2\}}$. Recall that $b+r = q+1$.

Now consider a vertex $O$ in $\Delta$ and a non-trivial elation $\alpha$ in $\PsiPsi 1(O)$, whose existence is ensured by Proposition~\ref{proposition:h1col}. Let $P \in \PP 1(O)$ and $\ell \in \LL 1(O)$ be the center and axis of the elation $\alpha$. Consider $m \in \LL 1(O)$ a line incident with $P$ but different from $\ell$. By Lemma~\ref{lemma:cycles}, the permutation induced by $\alpha$ on the set of $q$ points incident with $m$ but different from $P$ is a product of $k \geq 1$ cycles of length $c \geq 2$, with $k \cdot c = q$. If the chamber with vertices $O$, $P$ and $m$ is blue, then this implies that $b \equiv 1 \pmod c$ and $r \equiv 0 \pmod c$. If it is red, then $r \equiv 1 \pmod c$ and $b \equiv 0 \pmod c$. But $c$ does not depend on $m$, so this reasoning is valid for any choice of $m$. As $b$ cannot be congruent to both $0$ and $1$ modulo $c$ (because $c \geq 2$), this means that all the chambers with vertices $O$, $P$ and some $m \neq \ell$ have the same color. We can assume that this common color is blue, so that $b \equiv 1 \pmod c$ and $r \equiv 0 \pmod c$. In particular we have $r \geq c \geq 2$, but this is a contradiction with the fact that there is at most one red chamber adjacent to the panel defined by $O$ and $P$.
\end{proof}

\begin{remark}
The non-$4$-discreteness in Proposition~\ref{proposition:panel} can be replaced by non-$2$-discreteness. Indeed, in the proof we only need a non-trivial elation in $\PsiPsi 1(O)$ for some vertex $O$ (of any type), and Proposition~\ref{proposition:h1col} with $n+1$ instead of $n+3$ indeed gives a non-trivial $\hh 1$-collineation in $\PsiPsi n(O)$ for a vertex $O$ whose type is not controlled.

A similar remark can be done for many of our following results: we often assume that $\Aut(\Delta)$ is non-$f(n)$-discrete for some linear function $f$ of $n$, but we never claim that our choice for $f$ is optimal. In particular, the value $6n+2$ appearing in Theorem~\ref{maintheorem:A'} can certainly be replaced by a smaller value with some more effort.
\end{remark}

\subsection{From chamber-transitivity to \texorpdfstring{$1$-Moufangness}{1-Moufangness}}

The following theorem is due to Kantor \cite{Kantor} and concerns finite projective planes with a collineation group transitive on incident point-line pairs. This result will be helpful to get a local information about $\Aut(\Delta)$.

\begin{theorem}[Kantor, 1987]\label{theorem:Kantor}
Let $\Pi$ be a projective plane of order $q$, and let $F$ be a collineation group of $\Pi$ transitive on incident point-line pairs. Then either
\begin{enumerate}[(a)]
\item $\Pi$ is Desarguesian and $F \geq \PSL(3,q)$, or
\item $F$ is a Frobenius group of odd order $(q^2+q+1)(q+1)$, and $q^2+q+1$ is prime.
\end{enumerate}
\end{theorem}

\begin{proof}
See~\cite{Kantor}*{Theorem~A}.
\end{proof}

\begin{corollary}\label{corollary:Moufang}
Suppose that $\Aut(\Delta)^+$ is non-$4$-discrete and chamber-transitive. Then for each vertex $O$ in $\Delta$, the projective plane $\HH 1(O)$ is Desarguesian and $\PsiPsi 1(O) \geq \PSL(3,q)$, where $q+1$ is the number of chambers adjacent to each panel of $\Delta$. In particular, $\HH 1(O)$ is Moufang and $\PsiPsi 1(O)$ contains all elations of $\HH 1(O)$.
\end{corollary}

\begin{proof}
For any vertex $O$ in $\Delta$, $\HH 1(O)$ is a projective plane of order $q$. The chamber-transitivity of $\Delta$ directly implies that $\PsiPsi 1(O)$ is transitive on incident point-line pairs of $\HH 1(O)$. Hence, by Theorem~\ref{theorem:Kantor}, either $\HH 1(O)$ is Desarguesian and $\PsiPsi 1(O) \geq \PSL(3,q)$, or $|\PsiPsi 1(O)|\ = (q^2+q+1)(q+1)$. We only need to show that the latter is impossible. Note that there are exactly $(q^2+q+1)(q+1)$ incident point-line pairs in $\HH 1(O)$, so the equality $|\PsiPsi 1(O)|\ = (q^2+q+1)(q+1)$ would imply that the action of $\PsiPsi 1(O)$ on these point-line pairs is free. However, by Proposition~\ref{proposition:h1col}, there exists a non-trivial elation in $\PsiPsi 1(O)$. So the action is not free and the statement stands proven.

Note that, for a finite projective plane $\Pi$ (say of order $q$), being Desarguesian is equivalent to being Moufang. Also, in this case, the group generated by all elations of $\Pi$ is called the \textit{little projective group} and is exactly $\PSL(3,q)$.
\end{proof}

\subsection{From chamber-transitivity to Bruhat--Titsness}

We have seen with Corollary~\ref{corollary:Moufang} that all $\HH 1(O)$ are Moufang when $\Aut(\Delta)^+$ is chamber-transitive and non-$4$-discrete. Our next goal is to show, for each $n \geq 2$, that all $\HH n(O)$ are Moufang when $\Aut(\Delta)^+$ is chamber-transitive and non-$(6n+2)$-discrete.

We start with the next easy corollary of Proposition~\ref{proposition:h1col}.

\begin{lemma}\label{lemma:h1col}
Let $n \geq 1$ and suppose that $\Aut(\Delta)^+$ is non-$(n+3)$-discrete and chamber-transitive. Then for each vertex $O$ in $\Delta$, each point $P \in \PP n(O)$ and each line $\ell \in \LL n(O)$ with $P \ \II n \ \ell$, there exists a non-trivial $\hh 1$-collineation in $\PsiPsi n(O)$ with axis $\ell$ and center $P$.
\end{lemma}

\begin{proof}
By Proposition~\ref{proposition:h1col}, there exists a non-trivial $\hh 1$-collineation $\alpha \in \PsiPsi n(O)$, say with axis $\ell' \in \LL n(O)$ and center $P' \in \PP n(O)$. Let $c$ (resp.\ $c'$) be the chamber of $\Delta$ with vertices $O$, $\pipi 1(\ell)$ and $\pipi 1(P)$ (resp.\ $O$, $\pipi 1(\ell')$ and $\pipi 1(P')$). Since $\Aut(\Delta)^+$ is chamber-transitive, there exists $g \in \Aut(\Delta)^+$ such that $g(c) = c'$. Then $g^{-1} \alpha g$ is a non-trivial $\hh 1$-collineation, and by Lemma~\ref{lemma:changeaxis} it has axis $\ell$ and center $P$.
\end{proof}

\begin{lemma}\label{lemma:ball1}
Let $n \geq 2$ and let $1 \leq k < n$. In the following, $O$ is a vertex of $\Delta$, $P$ is a point in $\PP n(O)$ and $\ell$ is a line in $\LL n(O)$ with $P \ \II n \ \ell$, $Q$ is a point in $\PP n(O)$ not near $\ell$, and $m \in \LL n(O)$ is the line incident with $P$ and $Q$.
\begin{enumerate}[(i)]
\item Suppose that for any $O$, $P$ and $\ell$, there exists a non-trivial $\hh 1$-collineation in $\PsiPsi {2n+k}(O)$ with axis $\ell$ and center $P$. Then for any $O$, $P$ and $\ell$, there exists an elation $\alpha \in \PsiPsi n(O)$ with axis $\ell$ and center $P$ such that $\alpha^{\star_{k-1}}$ is trivial but $\alpha^{\star_k}$ is non-trivial.
\item Suppose that for any $O$, $P$, $\ell$ and $Q$, the group of all $\hh 1$-collineations in $\PsiPsi {2n+k}(O)$ with axis $\ell$ and center $P$ acts transitively on the set of points $(2n+k-1)$-neighboring $Q$ and incident with $m$. Then for any $O$, $P$, $\ell$ and $Q$, the group of all elations $\alpha \in \PsiPsi n(O)$ with axis $\ell$ and center $P$ and with $\alpha^{\star_{k-1}}$ trivial is transitive on the set of points in $\PP k(O)$ which are $(k-1)$-neighboring $\pipi k(Q)$ and incident with $\pipi k(m)$.
\end{enumerate}
\end{lemma}

\begin{proof}
Fix $O$, $\ell$ and $P$ and let $A$ be an appartment of $\Delta$ containing them (seen as vertices of $\Delta$). In $A$, we denote by $O'$ the reflection of $O$ over the line through $\ell$ and $P$ (see Figure~\ref{picture:ball1}). Also, $P'$ (resp.\ $\ell'$) is the vertex of $A$ at distance $2n+k$ from $O'$ such that $O'$ lies on the segment from $\ell$ to $P'$ (resp.\ from $P$ to $\ell'$). 

We first prove (i). By hypothesis, there exists a non-trivial $\hh 1$-collineation $\beta \in \PsiPsi {2n+k}(O')$ with axis $\ell'$ and center $P'$. We now consider the element $\alpha \in \PsiPsi n(O)$ induced by $\beta$. The fact that $\beta^{\star_{2n+k-1}}$ is trivial implies that $\alpha^{\star_{k-1}}$ is trivial. Also, it is clear from Lemma~\ref{lemma:12} (i) applied to $\beta$ that $\alpha^{\star_k}$ is non-trivial. There remains to show that $\alpha$ is an elation of $\HH n(O)$ with axis $\ell$ and center $P$, i.e.\ that $\alpha$ fixes all points incident with $\ell$ and all lines incident with $P$. This is actually also a consequence from the fact that $\beta^{\star_{2n+k-1}}$ (even $\beta^{\star_{2n}}$) is trivial. Indeed, all points incident with $\ell$ (and all lines incident with $P$) in $\HH n(O)$ correspond to vertices of $\Delta$ which are contained in $\HH{2n}(O')$ (more precisely in the convex hull of the vertices of $\Delta$ associated to $\PP{2n}(O')$ and $\LL{2n}(O')$).

\begin{figure}[t!]
\centering
\begin{pspicture*}(-4.5,-3.4)(4.5,7.6)
%\fontsize{11pt}{11pt}\selectfont
\psset{unit=0.7cm}

% Lines
\psline(0,0)(-2,3.464)
\psline(0,0)(2,3.464)
\psline(-2,3.464)(2,3.464)
\psline(-2,3.464)(0,6.928)
\psline(2,3.464)(0,6.928)
\psline(-2,3.464)(-5,8.66)
\psline(0,6.928)(-4,6.928)
\psline(0,6.928)(-2,10.392)
\psline(-2,6.928)(-1,8.66)
\psline[linestyle=dashed](0,0)(-5,0)
\psline[linestyle=dashed](0,0)(-2.5,-4.33)
\psline[linestyle=dashed](0,0)(2.5,-4.33)

% Chambres au milieu
\psline(-0.5,7.794)(-1.5,7.794)
\psline(-0.5,7.794)(-1.4,8.56)
\psline(-1.5,7.794)(-1.4,8.56)

% Chambres à gauche
\psline(-4.5,7.794)(-5.5,7.794)
\psline(-4.5,7.794)(-5.4,8.56)
\psline(-5.5,7.794)(-5.4,8.56)
\psline(-5.5,7.794)(-5,8.66)

% Fleches
\psline[arrows=<->](-0.6,0.25)(-2.2,3.0212)
\rput(-1.7,1.5){$n$}
\psline[arrows=<->](-2.6,3.614)(-4.2,6.4852)
\rput(-3.75,4.964){$n$}
\psline[arrows=<->](-3.6,7.1)(-4.4,8.5383)
\rput(-3.67,7.96){$k$}
\psline[arrows=<->](-0.6,-0.25)(-2.84,-4.13) % -1.6, -2.7712
\rput(-2.7,-2){$2n+k$}

% Notation
\psdot[linewidth=0.055](0,0)
\rput(0.55,0){$O'$} % (0,0)
\psdot[linewidth=0.055](0,6.928)
\rput(0.49,6.928){$O$} % (0,6.928)
\psdot[linewidth=0.055](-2,3.464)
\rput(-2.4,3.3){$P$} % (-2,3.464)
\psdot[linewidth=0.055](2,3.464)
\rput(2.3,3.3){$\ell$} % (2,3.464)
\psdot[linewidth=0.055](-4,6.928)
\rput(-4.39,6.8){$m$} % (-4,6.928)
\psdot[linewidth=0.055](-2,10.392)
\rput(-2.36,10.5){$Q$} % (-2,10.392)
\psdot[linewidth=0.055](-5,8.66)
\rput(-5.20,8.98){$Q'$} % (-5,8.66)
\psdot[linewidth=0.055](-5,0)
\rput(-5.39,0.1){$m'$} % (-5,0)
\psdot[linewidth=0.055](-2.5,-4.33)
\rput(-2.64,-4.67){$P'$} % (-2.5,-4.33)
\psdot[linewidth=0.055](2.5,-4.33)
\rput(2.7,-4.62){$\ell'$} % (2.5,-4.33)
\psdot[linewidth=0.055](-1,8.66)
\rput(-0.3,8.8){$\pipi k(Q)$} % (-1,8.66)
\psdot[linewidth=0.055](-2,6.928)
\rput(-2,6.55){$\pipi k(m)$} % (-2,6.928)
\end{pspicture*}
\caption{Illustration of Lemma~\ref{lemma:ball1}.}\label{picture:ball1}
\end{figure}

The reasoning is the same for (ii). Take $Q \in \PP n(O)$ a point not near $\ell$ and denote by $m \in \LL n(O)$ the line incident with $P$ and $Q$. Here also, we see $Q$ and $m$ as vertices of $\Delta$ and we can even assume that they belong to $A$. Let $Q'$ be the vertex of $A$ at distance $2n+k$ from $O'$, in the direction of $P$ and $m$. If $m' \in \LL{2n+k}(O')$ is the line incident with $P'$ and $Q'$ in $\HH{2n+k}(O')$, then the hypothesis states that the group of all $\hh 1$-collineations in $\PsiPsi{2n+k}(O')$ with axis $\ell'$ and center $P'$ acts transitively on the set of points $(2n+k-1)$-neighboring $Q'$ and incident with $m'$. Using Lemma~\ref{lemma:12} (ii) and as for (i), we obtain that the group of all elations $\alpha \in \PsiPsi n(O)$ with axis $\ell$ and center $P$ with $\alpha^{\star_{k-1}}$ trivial is transitive on the set of points in $\PP k(O)$ which are $(k-1)$-neighboring $\pipi k(Q)$ and incident with $\pipi k(m)$.
\end{proof}

The key result of this section is then the following.

\begin{proposition}\label{proposition:bigpiece}
Let $n \geq 1$ and suppose that $\Aut(\Delta)^+$ is non-$(2n+4)$-discrete and chamber-transitive. Let $O$ be a vertex in $\Delta$ and consider a point $P \in \PP n(O)$ and a line $\ell \in \LL n(O)$ with $P \ \II n \ \ell$. Let $Q \in \PP n(O)$ be a point not near $\ell$ and denote by $m \in \LL n(O)$ the line incident with $P$ and $Q$. Then the group of all $\hh 1$-collineations in $\PsiPsi n(O)$ with axis $\ell$ and center $P$ acts transitively on the set of points $(n-1)$-neighboring $Q$ and incident with $m$.
\end{proposition}

\begin{proof}
We introduce the three following assertions, all depending on $N \geq 1$ (actually $N \geq 2$ for $(C_N)$). Remark that $(A_n)$ is exactly what we need to prove.
\begin{enumerate}
\item[$(A_N)$] Let $O$ be a vertex in $\Delta$. Let $P \in \PP N(O)$ and $\ell \in \LL N(O)$ be such that $P \ \II N \ \ell$, let $Q \in \PP N(O)$ be a point not near $\ell$ and denote by $m$ the line incident with $P$ and $Q$. The group of all $\hh 1$-collineations in $\PsiPsi N(O)$ with axis $\ell$ and center $P$ acts transitively on the set of points $(N-1)$-neighboring $Q$ and incident with $m$. 
\item[$(B_N)$] Let $i, j, k$ be the three types of panels in some order and let $f$ be the word $ijkijkijk\!\ldots$ of length $2N$. Let $(c_0, c_1, \ldots, c_{2N})$ be a gallery of type $f$ in $\Delta$ (i.e.\ for each $1 \leq s \leq 2N$, the chambers $c_{s-1}$ and $c_s$ are adjacent and their common panel has type given by the $s^{\text{th}}$ letter of $f$). Then for any two chambers $d$ and $d'$ adjacent to both $c_0$ and $c_1$ (but different from them), there exists an automorphism of $\Delta$ fixing $c_0, c_1, \ldots, c_{2N}$ and sending $d$ to $d'$.
\item[$(C_N)$] Let $O$ be a vertex in $\Delta$. Let $P \in \PP N(O)$ and $\ell \in \LL N(O)$ be such that $P \ \II N \ \ell$, let $Q \in \PP N(O)$ be a point near $\ell$ but not neighboring $P$, and let $m, o \in \LL N(O)$ be two lines near $Q$ but not neighboring $\ell$. There exist a point $P' \in \PP N(O)$ $(N-1)$-neighboring $P$, a line $\ell' \in \LL N(O)$ neighboring $\ell$ (with $P' \ \II N \ \ell'$) and an elation in $\PsiPsi N(O)$ with axis $\ell'$ and center $P'$ sending $\pipi 1(m)$ to $\pipi 1(o)$.
\end{enumerate}

Note that $(A_1)$ is given by Corollary~\ref{corollary:Moufang}. It also follows from this corollary that $(B_1)$ is true. Indeed, if $O$ is the vertex of $\Delta$ adjacent to $c_0$, $c_1$ and $c_2$ (as defined in $(B_1)$), then having $\PsiPsi 1(O) \geq \mathrm{PSL}(3,q)$ implies the existence of an automorphism fixing $c_0, c_1, c_2$ and sending $d$ to $d'$. We now show different relations between $(A_N)$, $(B_N)$ and $(C_N)$.

\setcounter{claim}{0}

\begin{claim}\label{claim:A}
$(B_{N-1}) + (C_N) \Rightarrow (A_N)$ for each $2 \leq N \leq n$.
\end{claim}

\begin{claimproof*}
Let $O$, $P$, $\ell$, $Q$ and $m$ be as in $(A_N)$. Let also $R \in \PP N(O)$ be a point $(N-1)$-neighboring $Q$ and incident with $m$ (see Figure~\ref{picture:A}). We want to prove that there exists some $\hh 1$-collineation in $\PsiPsi N(O)$ with axis $\ell$ and center $P$, sending $Q$ to $R$.

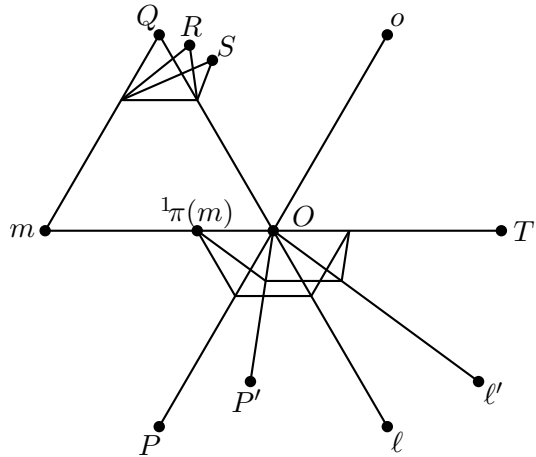
\begin{figure}[b!]
\centering
\begin{pspicture*}(-3.5,-3)(3.5,3)
%\fontsize{11pt}{11pt}\selectfont
\psset{unit=1cm}

% Lines
\psline(0,0)(-1.5,2.598)
\psline(0,0)(1.5,2.598)
\psline(0,0)(-3,0)
\psline(0,0)(3,0)
\psline(0,0)(-1.5,-2.598)
\psline(0,0)(1.5,-2.598)

\psline(0,0)(-0.3,-2)
\psline(0,0)(2.7,-2)

% Chambres au milieu
\psline(-1,0)(-0.5,-0.866)
\psline(-0.5,-0.866)(0.5,-0.866)
\psline(0.5,-0.866)(1,0)

\psline(-1,0)(-0.1,-0.666)
\psline(-0.1,-0.666)(0.9,-0.666)
\psline(0.9,-0.666)(1,0)

% Chambres à gauche
\psline(-1,1.732)(-2,1.732)
\psline(-3,0)(-1.5,2.598)
\psline(-1,1.732)(-1.1,2.46)
\psline(-2,1.732)(-1.1,2.46)
\psline(-1,1.732)(-0.8,2.26)
\psline(-2,1.732)(-0.8,2.26)

% Notation
\psdot[linewidth=0.04](0,0)
\rput(0.4,0.2){$O$} % (0,0)
\psdot[linewidth=0.04](-3,0)
\rput(-3.3,0){$m$} % (-3,0)
\psdot[linewidth=0.04](-1.5,-2.598)
\rput(-1.63,-2.84){$P$} % (-1.5,-2.598)
\psdot[linewidth=0.04](1.5,-2.598)
\rput(1.63,-2.8){$\ell$} % (1.5,-2.598)
\psdot[linewidth=0.04](3,0)
\rput(3.3,0){$T$} % (3,0)
\psdot[linewidth=0.04](1.5,2.598)
\rput(1.63,2.8){$o$} % (1.5,2.598)
\psdot[linewidth=0.04](-1.5,2.598)
\rput(-1.65,2.82){$Q$} % (-1.5,2.598)
\psdot[linewidth=0.04](-0.3,-2)
\rput(-0.35,-2.25){$P'$} % (-0.3,-2)
\psdot[linewidth=0.04](2.7,-2)
\rput(2.9,-2.15){$\ell'$} % (2.7,-2)
\psdot[linewidth=0.04](-1,0)
\rput(-1,0.24){$\pipi 1(m)$} % (-1,0)
\psdot[linewidth=0.04](-1.1,2.46)
\rput(-1.08,2.71){$R$} % (-1.1,2.46)
\psdot[linewidth=0.04](-0.8,2.26)
\rput(-0.62,2.45){$S$} % (-0.8,2.26)

\end{pspicture*}
\caption{Illustration of Proposition~\ref{proposition:bigpiece}, Claim~\ref{claim:A}.}\label{picture:A}
\end{figure}

By Lemma~\ref{lemma:h1col}, there exists a non-trivial $\hh 1$-collineation $\alpha \in \PsiPsi N(O)$ with axis $\ell$ and center $P$. (Note that $\Aut(\Delta)$ is non-$(N+3)$-discrete because $N+3 \leq n+3 \leq 2n+2$.) By Lemma~\ref{lemma:411easy}, $\alpha$ sends $Q$ to some $S \neq R$. We know from $(B_{N-1})$ that there exists $\beta \in \Aut(\Delta)^+$ fixing $Q$, $O$ and $\pipi 1(m)$ and sending $S$ to $R$. Then $\beta \alpha \beta^{-1}$ sends $Q$ to $R$ (as desired) and is a $\hh 1$-collineation with axis $\ell'$ and center $P'$, with $\pipi 1(P') \ \II 1 \ \pipi 1(m)$. Now there are two different cases:
\begin{itemize}
\item If $\pipi 1(\ell') = \pipi 1(\ell)$, then also $\pipi 1(P') = \pipi 1(P)$, and hence $\beta \alpha \beta^{-1}$ is a $\hh 1$-collineation with axis $\ell$ and center $P$ in view of Lemma~\ref{lemma:changeaxis}.
\item If $\pipi 1(\ell') \neq \pipi 1(\ell)$, then denote by $T \in \PP N(O)$ the point incident with $\ell$ and $\ell'$ and by $o \in \LL N(O)$ the line incident with $Q$ and $T$. By $(C_N)$, there exist a point $Q' \in \PP N(O)$ $(N-1)$-neighboring $Q$, a line $o' \in \LL N(O)$ neighboring $o$ (with $Q' \ \II N \ o'$) and an elation $\gamma \in \PsiPsi N(O)$ with axis $o'$ and center $Q'$ sending $\pipi 1(\ell')$ to $\pipi 1(\ell)$. Note that $\gamma$ fixes $Q$ and $R$ by Lemma~\ref{lemma:quasi-elation}. Thus $\gamma (\beta \alpha \beta^{-1}) \gamma^{-1}$ is a $\hh 1$-collineation with axis $\ell$ and center $P$, and it sends $Q$ to $R$. (Remark that this argument is valid when $\pipi 1(P') = \pipi 1(P)$, even though Figure~\ref{picture:A} does not represent that case.)\hfill $\blacksquare$
\end{itemize}
\end{claimproof*}

\begin{claim}\label{claim:B}
$(B_{N-1}) + (A_N) \Rightarrow (B_N)$ for each $N \geq 2$.
\end{claim}

\begin{claimproof}
Let $i, j, k$, $w$, $(c_0,c_1,\ldots, c_{2N})$, $d$ and $d'$ be as in $(B_N)$. We must find an automorphism of $\Delta$ fixing $c_0, \ldots, c_{2N}$ and sending $d$ to~$d'$.

By $(B_{N-1})$, we already have some $g \in \Aut(\Delta)^+$ fixing $c_0, \ldots, c_{2N-2}$ and sending $d$ to $d'$. Denote by $c'_{2N-1}$ the image of $c_{2N-1}$ by $g$. Now taking $O$, $P$, $\ell$ and $Q$ as in Figure~\ref{picture:B1}, we can apply $(A_N)$ to get an element $h \in \Aut(\Delta)^+$ fixing $c_0, \ldots, c_{2N-2}$ as well as $d$ and $d'$ and sending $c'_{2N-1}$ to $c_{2N-1}$. So $hg$ sends $d$ to $d'$ and fixes $c_0, \ldots, c_{2N-1}$. Now we can use the same method one step further: if $c'_{2N}$ denotes the image of $c_{2N}$ by $hg$, then we can find thanks to $(A_N)$ (see Figure~\ref{picture:B2}) an element $h'$ fixing $c_0, \ldots, c_{2N-1}$, $d$ and $d'$ and sending $c'_{2N}$ to $c_{2N}$. The element $h'hg$ then fixes $c_0, \ldots, c_{2N}$ and sends $d$ to $d'$.
\begin{figure}[h!]
\begin{subfigure}{.49\textwidth}

\centering
\begin{pspicture*}(-3.3,-3)(1.7,3.8)
%\fontsize{11pt}{11pt}\selectfont
\psset{unit=1cm}

% Lines
\psline(0,0)(-2,3.464)
\psline(0,0)(-1,0)
\psline(-1,0)(-2.5,2.598)
\psline(-2.5,2.598)(-2,3.464)
\psline(-1,0)(-0.5,0.866)
\psline(-0.5,0.866)(-1.5,0.866)
\psline[linestyle=dashed](-0.5,-0.866)(-1.5,-2.598)
\psline[linestyle=dashed](0,0)(1.5,-2.598)

% Chambres au milieu
\psline(0,0)(-0.5,-0.866)
\psline(-1,0)(-0.5,-0.866)
\psline(0,0)(-0.9,-0.766)
\psline(-1,0)(-0.9,-0.766)
\psline(0,0)(-1.2,-0.616)
\psline(-1,0)(-1.2,-0.616)

% Chambres à gauche
\psline(-2.5,2.598)(-1.5,2.598)
\psline(-2.5,2.598)(-1.6,3.314)
\psline(-1.5,2.598)(-1.6,3.314)

% Notation
\psdot[linewidth=0.04](0,0)
\rput(0.19,0.22){$O$} % (0,0)
\psdot[linewidth=0.04](-1.5,-2.598)
\rput(-1.62,-2.84){$P$} % (-1.5,-2.598)
\psdot[linewidth=0.04](1.5,-2.598)
\rput(1.61,-2.81){$\ell$} % (1.5,-2.598)
\rput(-1,-0.97){$d$}
\rput(-1.35,-0.8){$d'$}
\rput(-0.33,-1.05){$c_0$}
\rput(-0.5,0.26){$c_1$}
\rput(-1,0.53){$c_2$}
\psdot[linewidth=0.04](-2,3.464)
\rput(-2.22,3.6){$Q$} % (-2,3.464)
\rput(-2.77,3.1){$c_{2N-1}$}
\rput(-1.00,3){$c'_{2N-1}$}

\end{pspicture*}
\caption{From $c_{2N-1}$ to $c'_{2N-1}$.}\label{picture:B1}

\end{subfigure}
\begin{subfigure}{.49\textwidth}

\centering
\begin{pspicture*}(-1.7,-3)(2.9,3.8)
%\fontsize{11pt}{11pt}\selectfont
\psset{unit=1cm}

% Lines
\psline(-0.5,-0.866)(2,3.464)
\psline(0,0)(1,0)
\psline(0.5,-0.866)(2.5,2.598)
\psline(2.5,2.598)(2,3.464)
\psline(1,0)(0.5,0.866)
\psline(0.5,0.866)(1.5,0.866)
\psline(0,0)(0.5,-0.866)
\psline(-0.5,-0.866)(0.5,-0.866)
\psline[linestyle=dashed](0.5,-0.866)(1.5,-2.598)
\psline[linestyle=dashed](-0.5,-0.866)(-1.5,-2.598)

% Chambres au milieu
\psline(0,0)(-0.65,-0.55)
\psline(0.5,-0.866)(-0.65,-0.55)
\psline(0,0)(-0.6,-0.2)
\psline(0.5,-0.866)(-0.6,-0.2)

% Chambres à droite
\psline(2.5,2.598)(1.5,2.598)
\psline(2.5,2.598)(1.6,3.314)
\psline(1.5,2.598)(1.6,3.314)

% Notation
\psdot[linewidth=0.04](0,0)
\rput(-0.14,0.22){$O$} % (0,0)
\psdot[linewidth=0.04](-1.5,-2.598)
\rput(-1.63,-2.82){$\ell$} % (-1.5,-2.598)
\psdot[linewidth=0.04](1.5,-2.598)
\rput(1.63,-2.83){$P$} % (1.5,-2.598)
\rput(-0.82,-0.57){$d$}
\rput(-0.78,-0.14){$d'$}
\rput(0,-1.08){$c_0$}
\rput(0.5,-0.32){$c_1$}
\rput(0.5,0.26){$c_2$}
\rput(1,0.53){$c_3$}
\psdot[linewidth=0.04](2,3.464)
\rput(2.23,3.6){$Q$} % (2,3.464)
\rput(2.6,3.11){$c_{2N}$}
\rput(1.2,3){$c'_{2N}$}

\end{pspicture*}
\caption{From $c_{2N}$ to $c'_{2N}$.}\label{picture:B2}
\end{subfigure}
\caption{Illustration of Proposition~\ref{proposition:bigpiece}, Claim~\ref{claim:B}.}\label{picture:B}
\end{figure}
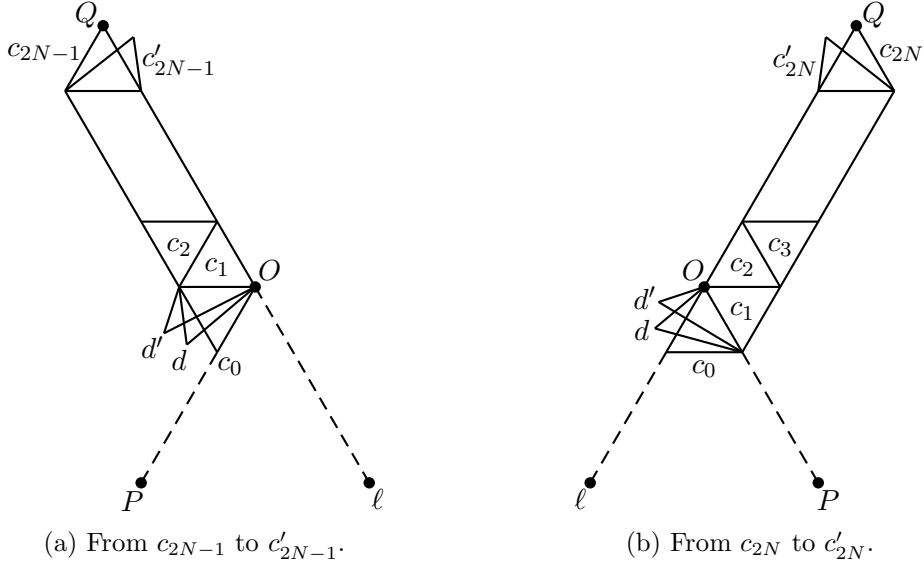
\end{claimproof}

\begin{claim}\label{claim:C}
$(B_{N-1}) \Rightarrow (C_N)$ for each $2 \leq N \leq n$.
\end{claim}

\begin{claimproof}
Let $O$, $P$, $\ell$, $Q$, $m$ and $o$ be as in $(C_N)$. We must find an elation in $\PsiPsi N(O)$ sending $\pipi 1(m)$ to $\pipi 1(o)$, with axis $\ell'$ and center $P'$ where $\ell'$ is neighboring $\ell$ and $P'$ is $(N-1)$-neighboring $P$.
\begin{figure}[t!]
\centering
\begin{pspicture*}(-2.5,-3)(3.7,3)
%\fontsize{11pt}{11pt}\selectfont
\psset{unit=1cm}

% Lines
\psline(0,0)(-1.5,-2.598)
\psline(0,0)(1.5,-2.598)
\psline(0,0)(3,0)
\psline(0,0)(1.5,2.598)

% Chambres au milieu
\psline(0.5,0.866)(1,0)
\psline(0,0)(2.55,2.148)
\psline(1,0)(0.85,0.716)
\psline(0,0)(3.3,1.548)
\psline(1,0)(1.1,0.516)

\psline(-0.5,-0.866)(0.5,-0.866)
\psline(-1.5,-2.598)(-0.5,-2.598)
\psline(-0.5,-2.598)(-1,-1.732)
\psline(-0.5,-2.598)(0.5,-0.866)
\psline(-1,-1.732)(0,-1.732)

% Notation
\psdot[linewidth=0.04](0,0)
\rput(-0.23,0.14){$O$} % (0,0)
\psdot[linewidth=0.04] (-1.5,-2.598)
\rput(-1.64,-2.86){$P$} % (-1.5,-2.598)
\psdot[linewidth=0.04](1.5,-2.598)
\rput(1.63,-2.83){$\ell$} % (1.5,-2.598)
\psdot[linewidth=0.04](3,0)
\rput(3.34,0){$Q$} % (3,0)
\psdot[linewidth=0.04](1.5,2.598)
\rput(1.63,2.82){$o$} % (1.5,2.598)
\psdot[linewidth=0.04](2.55,2.148)
\rput(2.79,2.33){$m$} % (2.55,2.148)
\psdot[linewidth=0.04](3.3,1.548)
\rput(3.55,1.58){$p$} % (3.3,1.548)
\psdot[linewidth=0.04](0.5,-0.866)
\rput(1,-0.82){$\pipi 1(\ell)$} % (0.5,-0.866)
\psdot[linewidth=0.04](-1,-1.732)
\rput(-1.82,-1.62){$\pipi {N-1}(P)$} % (-1,-1.732)

\end{pspicture*}
\caption{Illustration of Proposition~\ref{proposition:bigpiece}, Claim~\ref{claim:C}.}\label{picture:C}
\end{figure}
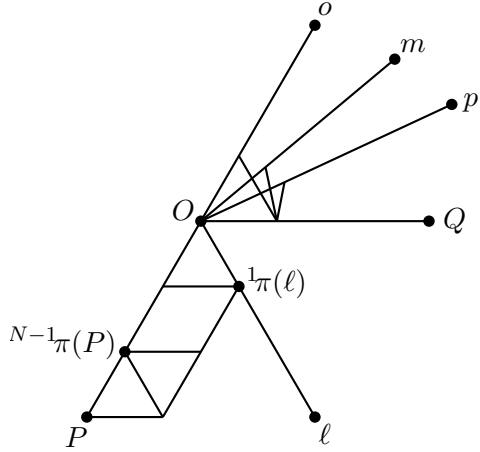

By Lemma~\ref{lemma:h1col} (with $2N+1$) and Lemma~\ref{lemma:ball1} (i) (with $N$), there exists an elation $\alpha \in \PsiPsi N(O)$ with axis $\ell$ and center $P$ and such that $\alpha^{\star_1}$ is non-trivial. In view of Lemma~\ref{lemma:411easy} (applied in $\HH 1(O)$), if $p$ denotes the image of $m$ by $\alpha$, then $\pipi 1(p) \neq \pipi 1(m)$. By $(B_{N-1})$, there exists $g \in \Aut(\Delta)^+$ fixing $\pipi{N-1}(P)$, $\pipi 1(\ell)$ and $\pipi 1(m)$ and sending $\pipi 1(p)$ to $\pipi 1(o)$ (see Figure~\ref{picture:C}). Then $g \alpha g^{-1}$ is an elation with axis $g(\ell)$ and center $g(P)$ which sends $\pipi 1(m)$ to $\pipi 1(o)$. Since $g(\ell)$ is neighboring $\ell$ and $g(P)$ is $(N-1)$-neighboring $P$, we are done.
\end{claimproof}

\medskip

Claims~\ref{claim:A} and~\ref{claim:C} together imply that $(B_{N-1}) \Rightarrow (A_N)$ for each $2 \leq N \leq n$ $(*)$, so that Claim~\ref{claim:B} then reads as $(B_{N-1}) \Rightarrow (B_N)$ for each $2 \leq N \leq n$. From $(B_1)$ we therefore get $(B_N)$ for all $1 \leq N \leq n$, and hence $(A_N)$ is true for all $2 \leq N \leq n$ by $(*)$. (Remember that $(A_1)$ was already true.)
\end{proof}

\begin{proof}[Proof of Theorem~\ref{maintheorem:A'}]
Suppose that $\Aut(\Delta)$ is non-$(6n+2)$-discrete. By Proposition~\ref{proposition:panel}, $\Aut(\Delta)^+$ is chamber-transitive (because $6n+2 \geq 4$). We want to prove that $\HH n(O)$ is Moufang for each vertex $O$ in $\Delta$.

Consider $P \in \PP n(O)$ and $\ell \in \LL n(O)$ with $P \ \II n \ \ell$. We need to show that $\HH n(O)$ is $(P,\ell)$-transitive. Let $m \in \LL n(O)$ be incident with $P$ but not neighboring $\ell$ and let $Q, R \in \PP n(O)$ be incident with $m$ but not neighboring $P$. We must find an elation of $\HH n(O)$ with axis $\ell$ and center $P$ sending $Q$ to $R$. We actually show by induction on $k$ that, for each $0 \leq k \leq n$, there exists an elation with axis $\ell$ and center $P$ sending $\pipi k(Q)$ to $\pipi k(R)$. For $k = 0$ we can take the identity (because $\pipi 0(Q) = \pipi 0(R) = O$ by convention). Now consider $1 \leq k \leq n$ and assume that this is true for $k-1$. Thus there is an elation $\alpha$ with axis $\ell$ and center $P$ such that $\alpha(\pipi{k-1}(Q)) = \pipi{k-1}(R)$. Denote by $Q'$ the image of $Q$ by $\alpha$. Then $Q'$ is $(k-1)$-neighboring $R$ and incident with $m$, and it suffices to find an elation with axis $\ell$ and center $P$ sending $\pipi k(Q')$ to $\pipi k(R)$. For $k = n$, such an elation exists by Proposition~\ref{proposition:bigpiece}, and for $k < n$ we need this same proposition (with $2n+k$) together with Lemma~\ref{lemma:ball1} (ii). (Note for Proposition~\ref{proposition:bigpiece} that $2(2n+k)+4 \leq 6n+2$ when $k \leq n-1$.)
\end{proof}

\section{Vertex-transitive \texorpdfstring{$\tilde{A}_2$-buildings}{triangle buildings}}

The goal of this section is to prove Theorem~\ref{maintheorem:B'} below.

\begin{primetheorem}{maintheorem:B}\label{maintheorem:B'}
Let $\Delta$ be a locally finite thick $\tilde{A}_2$-building. Suppose that $\Aut(\Delta)$ is transitive on vertices and unimodular, that $\Aut(\Delta)^+$ is transitive on vertices of each type, and that $\Delta$ has thickness $p+1$ for some prime $p$. Then for each $n \geq 1$, at least one of the following assertions holds:
\begin{enumerate}[(a)]
\item $\Delta$ is $n$-Moufang, or
\item $\Aut(\Delta)$ is $(6n+2)$-discrete.
\end{enumerate}
\end{primetheorem}

We will once again suppose that $\Aut(\Delta)$ is non-$(6n+2)$-discrete and, under the hypotheses of Theorem~\ref{maintheorem:B'}, prove that $\Aut(\Delta)^+$ must be transitive on panels of each type. The conclusion will then follow from Theorem~\ref{maintheorem:A'}.

\subsection{About finite projective planes}

We begin with several lemmas concerning finite projective planes. They will become useful later in the section. The first lemma is classical.

\begin{lemma}\label{lemma:transitive}
Let $\Pi$ be a finite projective plane and let $F$ be a collineation group of $\Pi$. Then $F$ is transitive on points of $\Pi$ if and only if $F$ is transitive on lines of $\Pi$.
\end{lemma}

\begin{proof}
It is actually true that, for any collineation group $F$ of a finite projective plane $\Pi$, $F$ has as many point orbits as line orbits, see~\cite{Hughes-Piper}*{Theorem~13.4}.
%Assume for a contradiction that $F$ is transitive on lines of $\Pi$ but not on points of $\Pi$. Then there exists a coloring of all points of $\Pi$ in two colors, say red and blue, such that $F$ preserves the coloring. Since $F$ is transitive on lines of $\Pi$, all lines of $\Pi$ contain the same number of red and blue points, say $r$ and $b$ respectively (with $r+b = q+1$, where $q$ is the order of $\Pi$). Now let us count the total number of red points. Let $P$ be a red point. There are $(q+1)$ lines though $P$, each containing $r$ red points. So we deduce that there are $(q+1)(r-1)+1$ red points in total. Now if we consider a blue point $Q$, there are $(q+1)$ lines through $Q$, each containing $r$ red points, and we deduce that there are $(q+1)\cdot r$ red points in total. This is a contradiction, since $(q+1)(r-1)+1 \neq (q+1)\cdot r$.
\end{proof}

The following lemma is also classical but, because of the lack in finding a suitable reference, we give its proof here.

\begin{lemma}\label{lemma:fixpoint}
Let $\Pi$ be a finite projective plane of prime order and let $F$ be a collineation group of $\Pi$. Suppose that $F$ contains a non-trivial elation. Then either $F$ is transitive on points of $\Pi$ or $F$ fixes a point or a line of $\Pi$.
\end{lemma}

\begin{proof}
We color the points of $\Pi$ according to their orbit under the action of $F$. Let us suppose that $F$ is not transitive on points of $\Pi$, i.e.\ that there are at least $2$ colors. Let us denote by $P$ and $\ell$ the center and axis of a non-trivial elation $\alpha$ in $F$. By Lemma~\ref{lemma:cycles} and since $\Pi$ has prime order, for each line $o$ incident to $P$ and different from $\ell$, the elation $\alpha$ is transitive on points incident to $o$ and different from $P$. Thus, for each such $o$, all points incident to $o$ and different from $P$ have the same color $(*)$. Now let us distinguish several cases:
\begin{itemize}
\item If $P$ has a color that no other point has, then $P$ is fixed by $F$.
\item Otherwise, and if the only points with the same color as $P$ are incident to $\ell$, then $\ell$ is fixed by $F$.
\item Now assume that there exists a point $P'$ not incident to $\ell$ but with the same color as $P$. This means that there exists $\beta \in F$ with $\beta(P) = P'$. Denote by $m$ the line through $P$ and $P'$, and write $\ell' = \beta(\ell)$. Note that, in view of $(*)$, for each line $o'$ incident to $P'$ and different from $\ell'$, all points incident to $o'$ and different from $P'$ have the same color $(**)$. See Figure~\ref{picture:fixpoint} for an illustration.
\begin{itemize}
\item If $\ell' = m$, we deduce from $(*)$, $(**)$ and the fact that $\beta(\ell) = \ell'$ that all points have the same color, which is a contradiction.
\item If $\ell' \neq m$, then we obtain from $(*)$ and $(**)$ that all points incident to $m$ have the same color, say $c_1$, and that all points not incident to $m$ but different from $Q = \ell \cap \ell'$ have the same color, say $c_2$. We write $c_3$ for the color of $Q$. If $c_3 \neq c_1, c_2$, then $Q$ is the only point with color $c_3$ so it is fixed by $F$. If $c_3 = c_2$, then $c_1 \neq c_2$ (because there are at least two colors), and $m$ is fixed by $F$. Finally, if $c_3 = c_1$, then $c_1 \neq c_2$ and there should exist $\gamma \in F$ with $\gamma(P) = Q$. But this gives a contradiction with the coloring. \qedhere
\end{itemize}
\end{itemize}
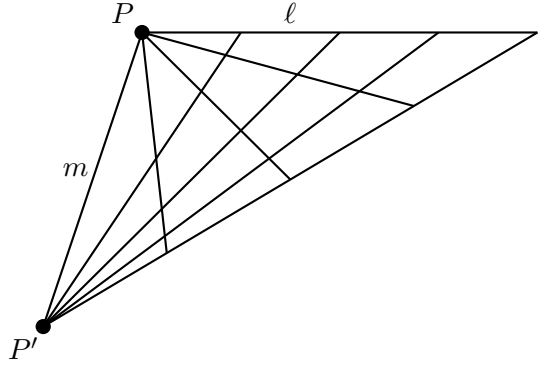
\begin{figure}[t!]
\centering
\begin{pspicture*}(-2,-4.4)(5.3,0.5)
%\fontsize{11pt}{11pt}\selectfont
\psset{unit=1.3cm}

% Lines
\psline(0,0)(4,0)
\psline(0,0)(-1,-3)
\psline(-1,-3)(1,0)
\psline(-1,-3)(2,0)
\psline(-1,-3)(3,0)
\psline(-1,-3)(4,0)
\psline(0,0)(0.25,-2.25)
\psline(0,0)(1.5,-1.5)
\psline(0,0)(2.75,-0.75)

% Dots
\psdot[linewidth=0.05](0,0)
\psdot[linewidth=0.05](-1,-3)

% Notation
\rput(-0.2,0.2){$P$}
\rput(-1.2,-3.22){$P'$}
\rput(-0.67,-1.4){$m$}
\rput(1.5,0.2){$\ell$}

\end{pspicture*}
\caption{Illustration of Lemma~\ref{lemma:fixpoint}.}\label{picture:fixpoint}
\end{figure}
\end{proof}

We conclude with a third lemma about finite projective planes of prime order which can be applied in some really precise situation.

\begin{lemma}\label{lemma:3colors}
Let $\Pi$ be a finite projective plane of prime order and $F$ be a collineation group of $\Pi$. Suppose that $F$ contains a non-trivial elation and that $F$ fixes exactly one point $Q$ and one line $m$, with $Q$ not incident to $m$. Then $F$ is transitive on points incident to $m$ and transitive on points not incident to $m$ but different from $Q$.
\end{lemma}

\begin{proof}
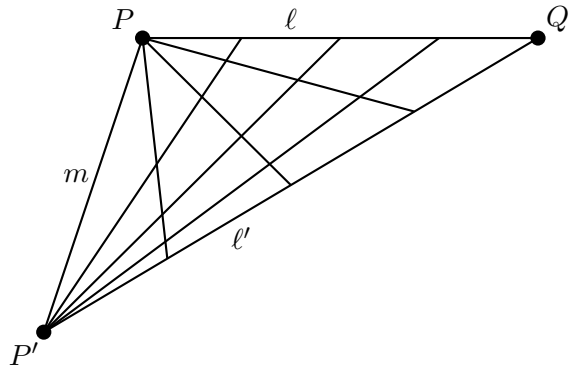
\begin{figure}[b!]
\centering
\begin{pspicture*}(-2,-4.4)(5.7,0.5)
%\fontsize{11pt}{11pt}\selectfont
\psset{unit=1.3cm}

% Lines
\psline(0,0)(4,0)
\psline(0,0)(-1,-3)
\psline(-1,-3)(1,0)
\psline(-1,-3)(2,0)
\psline(-1,-3)(3,0)
\psline(-1,-3)(4,0)
\psline(0,0)(0.25,-2.25)
\psline(0,0)(1.5,-1.5)
\psline(0,0)(2.75,-0.75)

% Dots
\psdot[linewidth=0.05](0,0)
\psdot[linewidth=0.05](4,0)
\psdot[linewidth=0.05](-1,-3)

% Notation
\rput(-0.2,0.2){$P$}
\rput(4.2,0.2){$Q$}
\rput(-1.2,-3.22){$P'$}
\rput(-0.67,-1.4){$m$}
\rput(1.5,0.2){$\ell$}
\rput(1,-2.04){$\ell'$}

\end{pspicture*}
\caption{Illustration of Lemma~\ref{lemma:3colors}.}\label{picture:3colors}
\end{figure}
Let $\alpha$ be a non-trivial elation in $F$, say with axis $\ell$ and center $P$. From Lemma~\ref{lemma:411easy}, we deduce that $Q$ is incident to $\ell$ (and different from $P$) and that $m$ is incident to $P$ (and different from $\ell$), see Figure~\ref{picture:3colors}. We color the points of $\Pi$ according to their orbit (under the action of $F$). By Lemma~\ref{lemma:cycles} and since $\Pi$ has prime order, for each line $o$ incident to $P$ and different from $\ell$, all points incident to $o$ and different from $P$ have the same color $(*)$. Now by hypothesis, $P$ is not fixed by $F$. Thus there exists $\beta \in F$ with $\beta(P) = P' \neq P$. Moreover, $P'$ is incident to $m$ since $m$ is fixed by $F$ (and hence by $\beta$). As $Q$ is fixed by $F$, we also have $\beta(\ell) = \ell'$ where $\ell'$ is the line incident to $P'$ and $Q$. So by $(*)$, we get that for each line $o'$ incident to $P'$ and different from $\ell'$, all points incident to $o'$ and different from $P'$ have the same color $(**)$. From $(*)$ and $(**)$ we deduce that there are exactly three colors: one for $Q$, one for the points incident to $m$ and one for all other points.
\end{proof}

\subsection{From vertex-transitivity to panel-transitivity}

We start this subsection with two easy lemmas.

\begin{lemma}\label{lemma:not-transitive}
Suppose that $\Aut(\Delta)^+$ is transitive on vertices of each type but not transitive on panels of each type. Then for each vertex $O$ in $\Delta$, $\PsiPsi 1(O)$ is not transitive on $\PP 1(O)$ (resp.\ $\LL 1(O)$).
\end{lemma}

\begin{proof}
Suppose for a contradiction that $\PsiPsi 1(O)$ is transitive on $\PP 1(O)$ for some vertex $O$ in $\Delta$, say of type $0$. By Lemma~\ref{lemma:transitive}, $\PsiPsi 1(O)$ is also transitive on $\LL 1(O)$. Since $\Aut(\Delta)^+$ is transitive on vertices of each type, this implies that $\Aut(\Delta)^+$ is transitive on panels of type $\{0,1\}$ and of type $\{0,2\}$ of $\Delta$. Now if we consider a vertex $O'$ of type $1$, then we know that the stabilizer of $O'$ in $\Aut(\Delta)^+$ is transitive on panels of type $\{0,1\}$ adjacent to $O'$. By Lemma~\ref{lemma:transitive}, it is also transitive on panels of type $\{1,2\}$ adjacent to $O'$. It follows that $\Aut(\Delta)^+$ is also transitive on panels of type $\{1,2\}$, which contradicts the hypothesis.
\end{proof}

\begin{lemma}\label{lemma:unimodular}
Suppose that $\Aut(\Delta)$ is transitive on vertices and unimodular. If $v$ and $w$ are two vertices in $\Delta$ such that the stabilizer $\Aut(\Delta)^+(v)$ of $v$ in $\Aut(\Delta)^+$ fixes $w$, then $\Aut(\Delta)^+(v) = \Aut(\Delta)^+(w)$.
\end{lemma}

\begin{proof}
We have $\Aut(\Delta)^+(v) \subseteq \Aut(\Delta)^+(w)$ by hypothesis. Take $g \in \Aut(\Delta)$ such that $g(v) = w$. Since $\Aut(\Delta)$ is unimodular, the Haar measure $\mu$ of $\Aut(\Delta)$ satisfies $\mu(\Aut(\Delta)^+(v)) = \mu(g \Aut(\Delta)^+(v) g^{-1}) = \mu(\Aut(\Delta)^+(w))$. This implies that $\Aut(\Delta)^+(v) = \Aut(\Delta)^+(w)$.
\end{proof}

The main result of this section is then the following.

\setcounter{claim}{0}

\begin{proposition}\label{proposition:bigone}
Suppose that $\Aut(\Delta)$ is transitive on vertices, non-$6$-discrete and unimodular, that $\Aut(\Delta)^+$ is transitive on vertices of each type, and that $\Delta$ has thickness $p+1$ for some prime $p$. Then $\Aut(\Delta)^+$ is transitive on panels of each type.
\end{proposition}

\begin{proof}
Let us assume for a contradiction that $\Aut(\Delta)^+$ is not transitive on panels of each type. By Lemma~\ref{lemma:not-transitive}, this implies that $\PsiPsi 1(v)$ is not transitive on $\PP 1(v)$ (and on $\LL 1(v)$) for each vertex $v$ in $\Delta$. In view of Lemmas~\ref{lemma:fixpoint} and~\ref{lemma:unimodular}, for each such $v$ there exists $w$ adjacent to $v$ in $\Delta$ such that $\Aut(\Delta)^+(v) = \Aut(\Delta)^+(w)$. From now on, we color in red all panels (i.e.\ edges) $[v,w]$ in $\Delta$ such that $\Aut(\Delta)^+(v) = \Aut(\Delta)^+(w)$. We have just seen that each vertex is adjacent to at least one red edge.

\begin{claim}\label{claim:trivial}
Let $v, w, x, y$ be vertices in $\Delta$, placed as shown below.
\begin{enumerate}[(i)]
\item If $[v,w]$ and $[v,x]$ are red, then $[w,x]$ is red.
\item If $[v,w]$ and $[v,y]$ are red, then $[v,x]$ is red.
\end{enumerate}
\centering
\begin{pspicture*}(-0.9,-0.3)(1.5,1.4)
%\fontsize{11pt}{11pt}\selectfont
\psset{unit=1.2cm}

% Chambers
\psline(0,0)(1,0)
\psline(0,0)(0.5,0.866)
\psline(0,0)(-0.5,0.866)
\psline(1,0)(0.5,0.866)
\psline(0.5,0.866)(-0.5,0.866)

% Notation
\psdot[linewidth=0.03](0,0)
\rput(-0.15,-0.15){$v$}
\psdot[linewidth=0.03](1,0)
\rput(1.15,-0.15){$w$}
\psdot[linewidth=0.03](0.5,0.866)
\rput(0.65,1.016){$x$}
\psdot[linewidth=0.03](-0.5,0.866)
\rput(-0.65,1.016){$y$}

\end{pspicture*}
\end{claim}

\begin{claimproof*}
The claim follows from the definition of a red edge:
\begin{enumerate}[(i)]
\item Having $[v,w]$ and $[v,x]$ red means that $\Aut(\Delta)^+(v) = \Aut(\Delta)^+(w)$ and that $\Aut(\Delta)^+(v) = \Aut(\Delta)^+(x)$, so $\Aut(\Delta)^+(w) = \Aut(\Delta)^+(x)$ and $[w,x]$ is red.
\item Having $[v,w]$ and $[v,y]$ red means that $\Aut(\Delta)^+(v) = \Aut(\Delta)^+(w)$ and that $\Aut(\Delta)^+(v) = \Aut(\Delta)^+(y)$. In particular, this implies that $\Aut(\Delta)^+(v)$ fixes $x$. By Lemma~\ref{lemma:unimodular}, this gives us $\Aut(\Delta)^+(v) = \Aut(\Delta)^+(x)$ so that $[v,x]$ is red. \hfill $\blacksquare$\vspace{0.2cm}
\end{enumerate}
\end{claimproof*}

\begin{claim}\label{claim:root}
Let $v$ be a vertex in $\Delta$ and let $\alpha$ be a non-trivial elation of $\HH 1(v)$, with axis $\ell$ and center $P$. Then all vertices $w$ adjacent to $v$ with $[v,w]$ red are incident to $P$ or $\ell$.
\end{claim}

\begin{claimproof}
This follows from Lemma~\ref{lemma:411easy}.
\end{claimproof}

\begin{claim}\label{claim:oppose}
For each vertex $v$ in $\Delta$, there exist two vertices $w, x$ adjacent to $v$ and opposite in $\HH 1(v)$ such that $[v,w]$ and $[v,x]$ are red.
\end{claim}

\begin{claimproof}
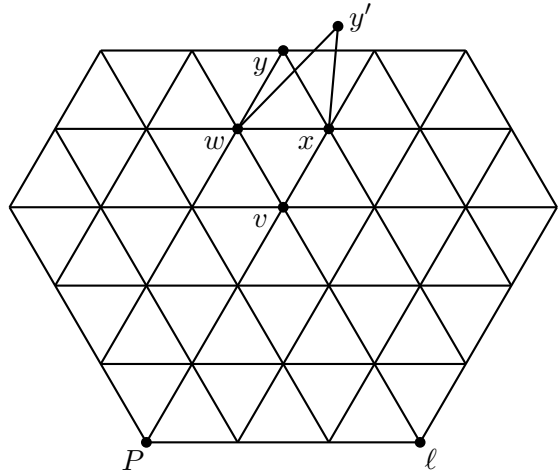
\begin{figure}[b!]
\centering
\begin{pspicture*}(-3.7,-3.5)(3.7,2.7)
%\fontsize{11pt}{11pt}\selectfont
\psset{unit=1.2cm}

% Chambers
\psline(-3,0)(3,0)
\psline(-2.5,-0.866)(2.5,-0.866)
\psline(-2,-1.732)(2,-1.732)
\psline(-1.5,-2.598)(1.5,-2.598)
\psline(-2.5,0.866)(2.5,0.866)
\psline(-2,1.732)(2,1.732)

\psline(-3,0)(-1.5,-2.598)
\psline(-2.5,0.866)(-0.5,-2.598)
\psline(-2,1.732)(0.5,-2.598)
\psline(-1,1.732)(1.5,-2.598)
\psline(0,1.732)(2,-1.732)
\psline(1,1.732)(2.5,-0.866)
\psline(2,1.732)(3,0)

\psline(3,0)(1.5,-2.598)
\psline(2.5,0.866)(0.5,-2.598)
\psline(2,1.732)(-0.5,-2.598)
\psline(1,1.732)(-1.5,-2.598)
\psline(-0,1.732)(-2,-1.732)
\psline(-1,1.732)(-2.5,-0.866)
\psline(-2,1.732)(-3,0)

\psline(-0.5,0.866)(0.6,2)
\psline(0.5,0.866)(0.6,2)

% Notation
\psdot[linewidth=0.03](0,0)
\rput(-0.25,-0.15){$v$}
\psdot[linewidth=0.03](-0.5,0.866)
\rput(-0.75,0.716){$w$}
\psdot[linewidth=0.03](0.5,0.866)
\rput(0.25,0.716){$x$}
\psdot[linewidth=0.03](0,1.732)
\rput(-0.25,1.55){$y$}
\psdot[linewidth=0.03](0.6,2)
\rput(0.85,2.1){$y'$}

\psdot[linewidth=0.03](-1.5,-2.598)
\rput(-1.65,-2.79){$P$}
\psdot[linewidth=0.03](1.5,-2.598)
\rput(1.62,-2.78){$\ell$}
\end{pspicture*}
\caption{Illustration of Proposition~\ref{proposition:bigone}, Claim~\ref{claim:oppose}.}\label{picture:claimoppose}
\end{figure}
By Lemmas~\ref{lemma:fixpoint} and~\ref{lemma:unimodular}, there is at least one red edge adjacent to any vertex. Since $\Aut(\Delta)$ is transitive on vertices, each vertex is adjacent to the same number of red edges. This number cannot be exactly one, because then there would be an issue with the types of the red panels (because $\Aut(\Delta)^+$ is transitive on vertices of each type). So each vertex is adjacent to at least two red edges.

We want to show that, for each vertex $v$, there exists $w, x$ adjacent to $v$ and opposite in $\HH 1(v)$ such that $[v,w]$ and $[v,x]$ are red. If this situation occurs at one vertex $v$, then it occurs at any vertex $v$ in view of the vertex-transitivity. So we assume for a contradiction that this situation does not appear anywhere.

First assume that, for some vertex $v$, there exist two vertices $w,y$ adjacent to $v$, with the same type and such that $[v,w]$ and $[v,y]$ are red. Then the edge $[v,x]$ between $w$ and $y$ must also be red, as well as $[w,x]$ and $[x,y]$ (by Claim~\ref{claim:trivial}). But there must also be two red edges of the same type adjacent to $w$. In all cases, we find (via Claim~\ref{claim:trivial}) two opposite red edges adjacent to a same vertex. So two such red edges $[v,w]$ and $[v,x]$ cannot exist, and the only remaining possibility is to have, for each vertex $v$ in $\Delta$, exactly two red edges adjacent to $v$, of different types and incident in $\HH 1(v)$ $(*)$.

We now show that this situation is impossible. Let us consider some non-trivial $\hh 1$-collineation $\alpha$ in $\HH 3(v)$, which exists by Proposition~\ref{proposition:h1col}. Denote by $P$ and $\ell$ its center and axis. Let $w, x$ be two vertices adjacent to $v$ in $\Delta$, placed as in Figure~\ref{picture:claimoppose}. Now for each vertex $y$ adjacent to both $w$ and $x$ but different from $v$, $\alpha$ induces an elation of $\HH 1(y)$ with axis $x$ and center $w$. Observing $(*)$ and Claim~\ref{claim:root} at $y$, we deduce that at least one of the edges $[y,w]$ and $[y,x]$ is red. This observation is true for any choice of $y$. If $p \geq 3$, there are at least three such vertices $y$ and we get two red edges $[w,y]$ and $[w,y']$ (or $[x,y]$ and $[x,y']$) with $y$ and $y'$ of the same type, which contradicts $(*)$. In the particular case where $p = 2$, we can also get a contradiction. First, if we denote by $y$ and $y'$ the two vertices adjacent to $w$ and $x$ and different from $v$, then the only way to not have a contradiction is to have $[w,y]$ and $[x,y']$ red (or $[w,y']$ and $[x,y]$ red). Now consider $x'$ a vertex adjacent to $v$ and $w$, different from $x$ and not adjacent to $\pipi 1(P)$. Then with the same argument as above we get two vertices $t$ and $t'$ adjacent to $w$ and $x'$ and such that $[w,t]$ and $[x,t']$ are red. This gives a contradiction with $(*)$ at $w$: the two edges $[w,y]$ and $[w,t]$ are red but $y$ and $t$ have the same type.
\end{claimproof}

\begin{claim}\label{claim:only2}
For each vertex $v$ in $\Delta$, there are exactly two red edges adjacent to $v$, and they are opposite in $\HH 1(v)$.
\end{claim}

\begin{claimproof}
\begin{figure}[b!]
\centering
\begin{pspicture*}(-3.7,-3.5)(3.7,2.9)
%\fontsize{11pt}{11pt}\selectfont
\psset{unit=1.2cm}

% Chambers
\psline(-3,0)(3,0)
\psline(-2.5,-0.866)(2.5,-0.866)
\psline(-2,-1.732)(2,-1.732)
\psline(-1.5,-2.598)(1.5,-2.598)
\psline(-2.5,0.866)(2.5,0.866)
\psline(-2,1.732)(2,1.732)

\psline(-3,0)(-1.5,-2.598)
\psline(-2.5,0.866)(-0.5,-2.598)
\psline(-2,1.732)(0.5,-2.598)
\psline(-1,1.732)(1.5,-2.598)
\psline(0,1.732)(2,-1.732)
\psline(1,1.732)(2.5,-0.866)
\psline(2,1.732)(3,0)

\psline(3,0)(1.5,-2.598)
\psline(2.5,0.866)(0.5,-2.598)
\psline(2,1.732)(-0.5,-2.598)
\psline(1,1.732)(-1.5,-2.598)
\psline(-0,1.732)(-2,-1.732)
\psline(-1,1.732)(-2.5,-0.866)
\psline(-2,1.732)(-3,0)

\psline(-0.5,0.866)(0.6,2)
\psline(0.5,0.866)(0.6,2)
\psline(-0.5,0.866)(-0.4,2)
\psline(0.5,0.866)(1.6,2)
\psline(-0.4,2)(1.6,2)

% Notation
\psdot[linewidth=0.03](0,0)
\rput(-0.25,-0.15){$v$}
\psdot[linewidth=0.03](-0.5,0.866)
\rput(-0.75,0.716){$w$}
\psdot[linewidth=0.03](0.5,0.866)
\rput(0.25,0.716){$x$}

\psdot[linewidth=0.03](0,1.732)
\rput(-0.25,1.55){$y$}
\psdot[linewidth=0.03](-1,1.732)
\rput(-1.23,1.58){$s$}
\psdot[linewidth=0.03](1,1.732)
\rput(0.75,1.58){$t$}

\psdot[linewidth=0.03](0.6,2)
\rput(0.8,2.25){$y'$}
\psdot[linewidth=0.03](-0.4,2)
\rput(-0.21,2.24){$s'$}
\psdot[linewidth=0.03](1.6,2)
\rput(1.79,2.24){$t'$}

\psdot[linewidth=0.03](-1.5,-2.598)
\rput(-1.65,-2.79){$P$}
\psdot[linewidth=0.03](1.5,-2.598)
\rput(1.62,-2.78){$\ell$}
\end{pspicture*}
\caption{Illustration of Proposition~\ref{proposition:bigone}, Claim~\ref{claim:only2}.}\label{picture:only2}
\end{figure}
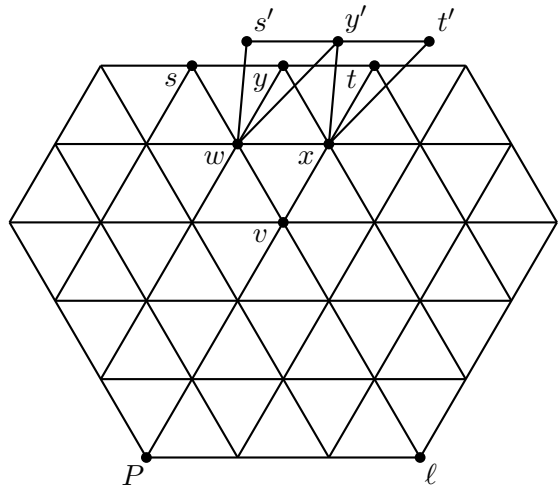
For each vertex $v$ in $\Delta$, we have two red edges adjacent to $v$ and opposite in $\HH 1(v)$, by Claim~\ref{claim:oppose}. Now assume for a contradiction that some (and hence any) vertex is adjacent to a third red edge.

For some vertex $v$, we consider some non-trivial $\hh 1$-collineation $\alpha$ in $\HH 3(v)$, with axis $\ell$ and center $P$. Let $w, x$ be two vertices adjacent to $v$ in $\Delta$, placed as in Figure~\ref{picture:only2}. Given a vertex $y$ adjacent to both $w$ and $x$ but different from $v$, $\alpha$ induces an elation of $\HH 1(y)$ with axis $x$ and center $w$. Applying Claims~\ref{claim:root} and~\ref{claim:oppose} at $y$, we obtain two red edges $[y,s]$ and $[y,t]$, with $s$ adjacent to $w$ and $t$ adjacent to $x$, see Figure~\ref{picture:only2}. We assumed that there is a third red edge $[y,r]$ adjacent to $y$. By Claim~\ref{claim:root}, $r$ must be adjacent to $w$ or $x$. Via Claim~\ref{claim:trivial}, this implies that all edges $[y,w]$, $[y,x]$, $[s,w]$, $[w,x]$ and $[x,t]$ are red. Now we can do the same reasoning with another vertex $y$ adjacent to $w$ and $x$ but different from $v$. This gives us two vertices $s'$ and $t'$ with $[y',s']$, $[y',t']$, $[y',w]$, $[y',x]$, $[s',w]$ and $[x,t']$ red. In particular, we get that the three edges $[w,s]$, $[w,x]$ and $[w,s']$ are red, with $s$, $x$ and $s'$ having the same type. In view of Claim~\ref{claim:root}, since there exists a non-trivial elation of $\HH 1(w)$, these three edges should be incident to a common edge. This is not the case, so we have our contradiction.
\end{claimproof}

\begin{claim}\label{claim:geodesic}
For each vertex $v$ in $\Delta$, there is a red bi-infinite geodesic through $v$.
\end{claim}

\begin{claimproof}
This follows directly from Claim~\ref{claim:only2}.
\end{claimproof}

\begin{claim}\label{claim:trapeze}
Let $v, w, x, y, z$ be vertices in $\Delta$ placed as shown below. If $[v,w]$ and $[v,x]$ are red, then $[y,z]$ is red.

\centering
\begin{pspicture*}(-1.5,-0.3)(1.5,1.4)
%\fontsize{11pt}{11pt}\selectfont
\psset{unit=1.2cm}

% Chambers
\psline(0,0)(-1,0)
\psline(0,0)(1,0)
\psline(0,0)(0.5,0.866)
\psline(0,0)(-0.5,0.866)
\psline(1,0)(0.5,0.866)
\psline(-1,0)(-0.5,0.866)
\psline(0.5,0.866)(-0.5,0.866)

% Notation
\psdot[linewidth=0.03](0,0)
\rput(0,-0.18){$v$}
\psdot[linewidth=0.03](1,0)
\rput(1.15,-0.15){$x$}
\psdot[linewidth=0.03](0.5,0.866)
\rput(0.65,1.016){$z$}
\psdot[linewidth=0.03](-0.5,0.866)
\rput(-0.65,1.016){$y$}
\psdot[linewidth=0.03](-1,0)
\rput(-1.15,-0.15){$w$}

\end{pspicture*}
\end{claim}

\begin{claimproof}
Consider some non-trivial $\hh 1$-collineation $\alpha$ of $\HH 2(v)$ given by Proposition~\ref{proposition:h1col} and denote by $P$ and $\ell$ its center and axis. Assume without loss of generality that the vertex $\pipi 1(P)$ (resp.\ $\pipi 1(\ell)$) has the same type as $x$ (resp.\ $w$). Recall from Claims~\ref{claim:only2} and~\ref{claim:geodesic} that there is a red bi-infinite geodesic through $w$, $v$ and $x$. We deduce that $w$ cannot be opposite to $\pipi 1(P)$ in $\HH 1(v)$, because then $\alpha$ would fix a line not near $P$, contradicting Lemma~\ref{lemma:411easy}. So $w$ must be adjacent to $\pipi 1(P)$. In the same way, we deduce that $x$ must be adjacent to $\pipi 1(\ell)$. Moreover, since $\Aut(\Delta)^+(v)$ fixes $w$ and $x$ and is transitive on points adjacent to $v$ and $w$ (by Lemma~\ref{lemma:3colors}), we can assume without loss of generality that $y$ and $z$ are different from $\pipi 1(P)$ and $\pipi 1(\ell)$, as in Figure~\ref{picture:trapeze}.

\begin{figure}[h!]
\centering
\begin{pspicture*}(-2.5,-2.5)(2.5,2.2)
%\fontsize{11pt}{11pt}\selectfont
\psset{unit=1.2cm}

% Chambers
\psline(-2,0)(2,0)
\psline(-1.5,-0.866)(1.5,-0.866)
\psline(-1,-1.732)(1,-1.732)
\psline(-1.5,0.866)(1.5,0.866)

\psline(-2,0)(-1,-1.732)
\psline(-1.5,0.866)(0,-1.732)
\psline(-0.5,0.866)(1,-1.732)
\psline(0.5,0.866)(1.5,-0.866)
\psline(1.5,0.866)(2,0)

\psline(2,0)(1,-1.732)
\psline(1.5,0.866)(0,-1.732)
\psline(0.5,0.866)(-1,-1.732)
\psline(-0.5,0.866)(-1.5,-0.866)
\psline(-1.5,0.866)(-2,0)

\psline(-1.5,0.866)(-1,1.732)
\psline(-1,1.732)(-0.5,0.866)

\psline(-0.5,0.866)(0.2,1.4)
\psline(0,0)(0.2,1.4)

% Notation
\psdot[linewidth=0.03](0,0)
\rput(-0.25,-0.15){$v$}
\psdot[linewidth=0.03](-1,0)
\rput(-1.25,-0.15){$w$}
\psdot[linewidth=0.03](1,0)
\rput(0.75,-0.15){$x$}
\psdot[linewidth=0.03](-0.5,0.866)
\rput(-0.75,0.7){$y$}
\psdot[linewidth=0.03](0.5,0.866)
\rput(0.25,0.716){$z$}

\psdot[linewidth=0.03](-1,1.732)
\rput(-1.28,1.64){$s_1$}
\psdot[linewidth=0.03](0.2,1.4)
\rput(-0.05,1.5){$s_2$}

\psdot[linewidth=0.03](-1,-1.732)
\rput(-1.15,-1.94){$P$}
\psdot[linewidth=0.03](1,-1.732)
\rput(1.12,-1.94){$\ell$}
\end{pspicture*}
\caption{Illustration of Proposition~\ref{proposition:bigone}, Claim~\ref{claim:trapeze}.}\label{picture:trapeze}
\end{figure}
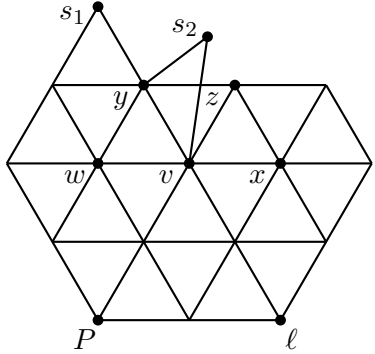

We now prove that $[y,z]$ is red. We already know by the previous claims that there is a (unique) vertex $s$ adjacent to $y$ and with the same type as $z$ such that $[y,s]$ is red. Our goal is to show that $s = z$. First observe that $s$ cannot be opposite to $v$ in $\HH 1(y)$ (as $s_1$ in Figure~\ref{picture:trapeze}). Indeed, if this was the case, then it would mean that $\alpha$ fixes $s$, a point of $\HH 2(v)$ not near $P$. This is impossible by Lemma~\ref{lemma:411easy}. So $s$ is adjacent to $v$.

Of course we cannot have $s = w$ since $[w,v]$ and $[w,y]$ cannot be both red. In order to show that $s = z$, there remains to show that $s$ is adjacent to $x$. We proceed by contradiction, assuming that $s$ is not adjacent to $x$ (as $s_2$ in Figure~\ref{picture:trapeze}). We thus have a red edge $[y,s]$ with $y$ and $s$ adjacent to $v$, $y$ adjacent to $w$ but $s$ not adjacent to $x$. In the case where $p \geq 3$, the contradiction will come from Lemma~\ref{lemma:3colors}. Indeed, if we denote by $Y$ the set of vertices adjacent to $v$ and $w$, and by $S$ the set of vertices with the same type as $s$, adjacent to $v$ and not adjacent to $x$, then Lemma~\ref{lemma:3colors} tells us that $\Aut(\Delta)^+(v)$ is transitive on $Y$ and on $S$. But $|Y|\ = p+1$ and $|S|\ = p^2-1$, so if $p \geq 3$ then having a red edge $[y,s]$ from a vertex in $Y$ to a vertex in $S$ implies that each vertex in $Y$ has more than one red edge going to a vertex in $S$. This is impossible, as $s$ is the only vertex of that type with $[y,s]$ red.

Let us now consider the last case $p = 2$. We continue our proof by contradiction, assuming that $s \neq z$. This time we have $|Y|\ = 3 = |S|$, and each vertex in $Y$ is adjacent to a unique vertex in $S$. This gives us three red edges. If we do the same reasoning around $z$ instead of $y$, then we denote by $Z$ the set of vertices adjacent to $v$ and $x$, by $S'$ the set of vertices with the same type as $y$, adjacent to $v$ and not adjacent to $w$, and we get three other red edges, each one connecting a vertex of $Z$ and a vertex of $S'$. In total, we got six red edges connecting neighbors of $v$. Now since $\Aut(\Delta)$ is transitive on vertices, this whole situation around $v$ also occurs around $w$. If we denote by $a$ the vertex adjacent to $w$ such that $[w,a]$ is red (with $a \neq v$), this means that $[y,b]$ is red, where $b$ is the unique vertex adjacent to $w$ and $y$, different from $v$ and not adjacent to $a$ (see Figure~\ref{picture:trapeze2}). But then, around $y$, we have $[y,b]$ and $[y,s]$ red, while $[w,v]$ is also red. This situation is different from the one around $v$, so we get our contradiction.
\begin{figure}[t!]
\centering
\begin{pspicture*}(-2.8,-0.3)(1.3,2)
%\fontsize{11pt}{11pt}\selectfont
\psset{unit=1.2cm}

% Chambers
\psline(-2,0)(1,0)
\psline(-1.5,0.866)(0.5,0.866)
\psline(-2,0)(-1.5,0.866)
\psline(-1.5,0.866)(-1,0)
\psline(-1,0)(-0.5,0.866)
\psline(-0.5,0.866)(0,0)
\psline(0,0)(0.5,0.866)
\psline(0.5,0.866)(1,0)

\psline(-0.5,0.866)(0.2,1.4)
\psline(0,0)(0.2,1.4)

\psline(-1,0)(-1.2,1.4)
\psline(-0.5,0.866)(-1.2,1.4)

% Notation
\psdot[linewidth=0.03](0,0)
\rput(-0.25,-0.15){$v$}
\psdot[linewidth=0.03](-1,0)
\rput(-1.25,-0.15){$w$}
\psdot[linewidth=0.03](1,0)
\rput(0.75,-0.15){$x$}
\psdot[linewidth=0.03](-0.5,0.866)
\rput(-0.75,0.7){$y$}
\psdot[linewidth=0.03](0.5,0.866)
\rput(0.25,0.716){$z$}

\psdot[linewidth=0.03](-2,0)
\rput(-2.2,-0.15){$a$}

\psdot[linewidth=0.03](-1.2,1.4)
\rput(-1.4,1.55){$b$}
\psdot[linewidth=0.03](0.2,1.4)
\rput(0.38,1.55){$s$}
\end{pspicture*}
\caption{Illustration of Proposition~\ref{proposition:bigone}, Claim~\ref{claim:trapeze}.}\label{picture:trapeze2}
\end{figure}
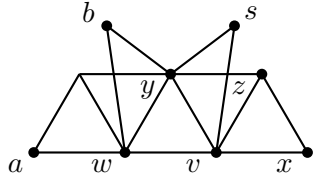
\end{claimproof}

\medskip

We now find a new contradiction. This will show that our hypotheses were wrong since the beginning, i.e.\ that $\Aut(\Delta)^+$ must be transitive on panels of each type.

Fix a vertex $v$ in $\Delta$ and consider a non-trivial $\hh 1$-collineation $\alpha$ of $\HH 2(v)$ given by Proposition~\ref{proposition:h1col}, say with axis $\ell$ and center $P$. We choose a vertex $w$ adjacent to $v$ and $\pipi 1(P)$ but different from $\pipi 1(\ell)$ and a vertex $x$ adjacent to $w$ and $v$ but different from $\pipi 1(P)$, as shown in Figure~\ref{picture:finalabsurd}. The $\hh 1$-collineation $\alpha$ induces a non-trivial elation of $\HH 1(x)$ with axis $v$ and center $w$. By Claim~\ref{claim:root}, this implies that the two red edges adjacent to $x$ (given by Claim~\ref{claim:only2}) are incident to $w$ and $v$ in $\HH 1(x)$. Hence, we conclude via Claim~\ref{claim:trapeze} that $[v,w]$ is also red. However, this reasoning could be done for any choice of $w$. So if $w'$ is another vertex adjacent to $v$ and $\pipi 1(P)$ but different from $\pipi 1(\ell)$, then we also get that $[v,w']$ is red. This gives a contradiction with Claim~\ref{claim:only2}.
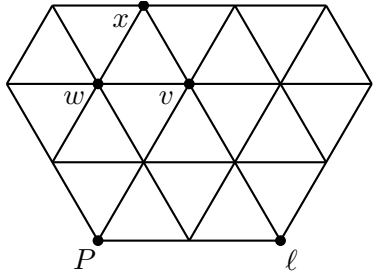
\begin{figure}[h!]
\centering
\begin{pspicture*}(-2.5,-2.5)(2.5,1.1)
%\fontsize{11pt}{11pt}\selectfont
\psset{unit=1.2cm}

% Chambers
\psline(-2,0)(2,0)
\psline(-1.5,-0.866)(1.5,-0.866)
\psline(-1,-1.732)(1,-1.732)
\psline(-1.5,0.866)(1.5,0.866)

\psline(-2,0)(-1,-1.732)
\psline(-1.5,0.866)(0,-1.732)
\psline(-0.5,0.866)(1,-1.732)
\psline(0.5,0.866)(1.5,-0.866)
\psline(1.5,0.866)(2,0)

\psline(2,0)(1,-1.732)
\psline(1.5,0.866)(0,-1.732)
\psline(0.5,0.866)(-1,-1.732)
\psline(-0.5,0.866)(-1.5,-0.866)
\psline(-1.5,0.866)(-2,0)

% Notation
\psdot[linewidth=0.03](0,0)
\rput(-0.25,-0.15){$v$}
\psdot[linewidth=0.03](-1,0)
\rput(-1.26,-0.15){$w$}
\psdot[linewidth=0.03](-0.5,0.866)
\rput(-0.75,0.7){$x$}

\psdot[linewidth=0.03](-1,-1.732)
\rput(-1.15,-1.94){$P$}
\psdot[linewidth=0.03](1,-1.732)
\rput(1.12,-1.94){$\ell$}
\end{pspicture*}
\caption{Illustration of Proposition~\ref{proposition:bigone}.}\label{picture:finalabsurd}
\end{figure}
\end{proof}

Theorem~\ref{maintheorem:B'} now follows immediately.

\begin{proof}[Proof of Theorem~\ref{maintheorem:B'}]
See Proposition~\ref{proposition:bigone} and Theorem~\ref{maintheorem:A'}.
\end{proof}

\section{A sufficient condition for exoticity}

In this section we prove Theorem~\ref{maintheorem:exotic'}, which gives a sufficient condition under which an $\tilde{A}_2$-building is not $2$-Moufang (and in particular exotic).

\begin{primetheorem}{maintheorem:exotic}\label{maintheorem:exotic'}
Let $\Delta$ be a locally finite thick $\tilde{A}_2$-building, let $x_0$ and $x_1$ be two adjacent vertices in $\Delta$ and let $C$ be the set of chambers adjacent to both $x_0$ and $x_1$. For each $j \in \{0,1\}$, let $G_j \leq \Sym(C)$ be the image of $\Aut(\HH 1(x_j))(x_{1-j})$ in $\Sym(C)$. If $G_0 \neq G_1$, then $\Delta$ is not $2$-Moufang.
\end{primetheorem}

\begin{proof}
Say that $\Delta$ has thickness $q+1$, i.e.\ $|C|\ = q+1$. Then $\HH 1(x_0)$ and $\HH 1(x_1)$ are projective planes of order $q$. If one of them is non-Desarguesian then $\Delta$ is not $1$-Moufang (in particular not $2$-Moufang), so we can assume that $q$ is a prime power and that they are both Desarguesian. The full automorphism group of the Desarguesian projective plane of order~$q$ is $\mathrm{P\Gamma L}(3,q)$, and the stabilizer of a line acts on the set of points incident to it as $\mathrm{P\Gamma L}(2,q)$ acting on the projective line over $\mathbf{F}_q$. So for each $t \in \{0,1\}$, the image $G_t \leq \Sym(C)$ of $\Aut(\HH 1(x_j))(x_{1-j})$ in $\Sym(C)$ is conjugate to $\mathrm{P\Gamma L}(2,q)$ in $\Sym(C)$. Let us now suppose for a contradiction that $\Delta$ is $2$-Moufang. 

The subgroup $\mathcal{G}_0$ of $\Aut(\HH 1(x_0))$ generated by all elations of $\HH 1(x_0)$ (i.e.\ the \textit{little projective group} of $\HH 1(x_0)$) is isomorphic to $\mathrm{PSL}(3,q)$. The image $G'_0 \leq \Sym(C)$ of $\mathcal{G}_0(x_{1})$ in $\Sym(C)$ is thus conjugate to $\mathrm{PGL}(2,q)$ (acting on the projective line over $\mathbf{F}_q$) in $\Sym(C)$. Now the fact that $\HH 2(x_0)$ is Moufang implies that each elation of $\HH 1(x_0)$ is the restriction of an elation of $\HH 2(x_0)$. We thus deduce that the image of $\Aut(\HH 2(x_0))(x_1)$ in $\Sym(C)$ contains $G'_0$, while being contained in $G_0$ and $G_1$. But $G_0 \cong \mathrm{P\Gamma L}(2,q)$ has only one subgroup that is conjugate to $\mathrm{PGL}(2,q)$ in $\Sym(C)$, and it is the normalizer of that subgroup, so $G_0 = N_{\Sym(C)}(G'_0)$. The same is true for $G_1$, so $G_1 = N_{\Sym(C)}(G'_0) = G_0$. This contradicts the hypothesis.
\end{proof}

%\begin{proof}[Proof of Theorem~\ref{maintheorem:exotic}]
%This follows directly from Theorem~\ref{maintheorem:exotic'}.
%\end{proof}

\section{Singer cyclic lattices}\label{section:Singer}

Let us now focus on Singer cyclic lattices, i.e.\ groups $\Gamma \leq \Aut(\Delta)$ acting simply transitively on the panels of each type of an $\tilde{A}_2$-building $\Delta$ and with the additional property that vertex-stabilizers are cyclic. These lattices have been deeply studied by Essert and Witzel in~\cite{Essert} and~\cite{Witzel}. The notion of a \textit{difference matrix} was defined in the latter reference. For our purpose, we present another way of understanding the relation between difference matrices and Singer cyclic lattices.

A \textbf{difference set} with parameter $q$ is a subset $D = \{d_1, \ldots, d_{q+1}\}$ of $\Z/(q^2+q+1)\Z$ such that, for each $x \in \Z/(q^2+q+1)\Z$ with $x \neq 0$, there exists a unique ordered pair $(d,d') \in D^2$ satisfying $x = d-d'$. Given such a difference set $D$ with parameter $q$, we can construct a projective plane $\Pi_D$ of order $q$ as follows. The point set $P$ and line set $L$ of $\Pi_D$ are simply $P = L = \Z/(q^2+q+1)\Z$, and the incidence relation $R \subseteq L \times P$ is given by
$$R = \{(x,x+d) \mid x \in L, d \in D\}.$$
It is an easy task to check that this defines a projective plane of order $q$.

Define a \textbf{difference vector} with parameter $q$ as a vertical vector $v = (d_1, \ldots, d_{q+1})^T$ where $\{d_1, \ldots, d_{q+1}\}$ is a difference set. To such a difference vector $v$, we associate a \textit{labelled projective plane}. A \textbf{labelled projective plane} of order $q$ is a projective plane whose flags (i.e.\ incident point-line pairs) are labelled by $\{1,\ldots,q+1\}$, i.e.\ with a map $\ell \colon R \to \{1,\ldots,q+1\}$. Given a difference vector $v$, we take the projective plane $\Pi_D$ associated to the difference set $D$ inherent to $v$, and we label its flags by defining $\ell(x,x+d_j) = j$ for each $x \in L$ and each $j \in \{1, \ldots, q+1\}$. Note that we need a difference vector (and not only a difference set) for this map to be well defined. We call $\Pi_v$ this labelled projective plane associated to $v$.

Now a \textbf{difference matrix} with parameter $q$ is a matrix with $q+1$ lines and $3$ columns, such that each of the three columns is a difference vector with parameter $q$. Let us write $M = (v_0, v_1, v_2)$ for such a matrix, where $v_0$, $v_1$ and $v_2$ are difference vectors. To a difference matrix $M$, we associate a \textbf{labelled $\tilde{A}_2$-building}, i.e.\ an $\tilde{A}_2$-building whose chambers are labelled by $\{1, \ldots, q+1\}$. Note that at each vertex of a labelled $\tilde{A}_2$-building, we see a labelled projective plane. (At a vertex of type $t \in \{0,1,2\}$, we consider vertices of type $t+1 \bmod 3$ as points and those of type $t+2 \bmod 3$ as lines). The labelled $\tilde{A}_2$-building $\Delta_M$ associated to the difference matrix $M = (v_0, v_1, v_2)$ is then defined as the unique one whose labelled projective plane at each vertex of type $t$ is $\Pi_{v_t}$ (for each $t \in \{0,1,2\}$). This building can be constructed recursively with the method of \cite{Ronan}: the labellings of the projective planes exactly tells us how two adjacent projective planes must be glued in the building. Moreover, we can define $\Gamma_M \leq \Aut(\Delta_M)$ as the group of all type-preserving automorphisms of $\Delta_M$ preserving the labellings. It is a direct fact that $\Gamma_M$ acts simply transitively on the panels of each type of $\Delta_M$ and that vertex stabilizers in $\Gamma_M$ are cyclic (of order $q^2+q+1$). So $\Gamma_M$ is a Singer cyclic lattice. Conversely, given a Singer cyclic lattice $\Gamma \leq \Aut(\Delta)$ we can label the chambers of $\Delta$ according to their orbit under the action of $\Gamma$ and get a (not necessarily unique) difference matrix $M$ such that $\Gamma = \Gamma_M$ and $\Delta = \Delta_M$. %The relation between difference matrices and Singer cyclic lattices is explained in~\cite{Witzel} with a different approach.

Two Singer cyclic lattices $\Gamma \leq \Aut(\Delta)$ and $\Gamma' \leq \Aut(\Delta')$ are \textbf{isomorphic} if there exists an isomorphism from $\Delta$ to $\Delta'$ conjugating $\Gamma$ to $\Gamma'$. This is actually equivalent to saying that $\Gamma$ and $\Gamma'$ are isomorphic as groups (see~\cite{Witzel}*{Proposition~3.7}). Two difference matrices $M$ and $M'$ are then said to be \textbf{equivalent} if $\Gamma_M \leq \Aut(\Delta_M)$ and $\Gamma_{M'} \leq \Aut(\Delta_{M'})$ are isomorphic. This equivalence relation on difference matrices was deeply studied in~\cite{Witzel}. In order to prove Corollary~\ref{maincorollary:bound} we do not need to really study the notion of equivalent difference matrices. We will however need the following basic results which can also be found in~\cite{Witzel}. A difference set $D$ (resp.\ difference vector $v$) is called \textbf{Desarguesian} if $\Pi_D$ (resp.\ $\Pi_v$) is Desarguesian. A difference matrix $M = (v_0, v_1, v_2)$ is called \textbf{Desarguesian} if $v_0$, $v_1$ and $v_2$ are Desarguesian. Note that there exist Desarguesian difference sets with parameter $q$ for each prime power $q$, see~\cite{Singer} or \cite{Witzel}*{Theorem~2.2}.

\begin{lemma}\label{lemma:equivalence}
Let $q = p^\eta$, with $p$ prime and $\eta \geq 1$.
\begin{enumerate}[(i)]
\item Let $M$ be a difference matrix with parameter $q$ and let $M'$ be a difference matrix obtained by permuting the $q+1$ lines of $M$. Then $M$ and $M'$ are equivalent.
\item Let $M = (v_0, v_1, v_2)$ be a difference matrix with parameter $q$, and let $g_0, g_1, g_2 \in \mathrm{AGL}(1,\Z/(q^2+q+1)\Z)$. Then $M$ is equivalent to $M' = (g_0(v_0), g_1(v_1), g_2(v_2))$, where $g_t$ acts on the difference vector $v_t$ componentwise.
\item Let $D$ be a Desarguesian difference set with parameter $q$. The stabilizer of $D$ in $\mathrm{AGL}(1,\Z/(q^2+q+1)\Z)$ has order $3\eta$.
\item Let $D$ be a Desarguesian difference set and $M$ be a Desarguesian difference matrix (both with parameter $q$). Then $M$ is equivalent to a difference matrix whose columns are equal to $D$ as a set.
\end{enumerate}
\end{lemma}

\begin{proof}
\begin{enumerate}[(i)]
\item Permuting the lines of a difference matrix $M = (v_0, v_1, v_2)$ simply permutes the labels in the three labelled projective planes $\Pi_{v_0}$, $\Pi_{v_1}$ and $\Pi_{v_2}$ simultaneously. So the labelled $\tilde{A}_2$-buildings $\Delta_M$ and $\Delta_{M'}$ are equal, up to permuting the labels. In particular, $\Gamma_M \leq \Aut(\Delta_M)$ and $\Gamma_{M'} \leq \Aut(\Delta_{M'})$ are isomorphic.
\item When $g \in \mathrm{AGL}(1,\Z/(q^2+q+1)\Z)$ and $v$ is a difference vector with parameter $q$, the labelled projective planes $\Pi_v$ and $\Pi_{g(v)}$ are isomorphic. Replacing a column $v_t$ by $g_t(v_t)$ thus does not change the Singer cyclic lattice.
\item See~\cite{Berman} or~\cite{Witzel}*{Lemma~4.5}.
\item This follows from (ii) and the fact that $\mathrm{AGL}(1,\Z/(q^2+q+1)\Z)$ is transitive on the Desarguesian difference sets with parameter $q$, see~\cite{Berman}. \qedhere
\end{enumerate}
\end{proof}

We can now prove Corollary~\ref{maincorollary:bound'} below.

\begin{primecorollary}{maincorollary:bound}\label{maincorollary:bound'}
For each $q \geq 2$, there are at most $\left(\frac{q(q^2-1)}{3}\right)^2$ isomorphism classes of Singer cyclic lattices $\Gamma \leq \Aut(\Delta)$ with parameter $q$ such that $\Delta$ is $2$-Moufang.
\end{primecorollary}

\begin{proof}
If $q$ is not a prime power then the claim is obvious: an $\tilde{A}_2$-building with thickness $q+1$ is never $1$-Moufang when $q$ is not a prime. We now assume that $q = p^\eta$ and fix some Desarguesian difference set $D = \{d_1, \ldots, d_{q+1}\}$ with parameter $q$. We need an upper bound on the number of equivalence classes of difference matrices $M$ with parameter $q$ such that $\Delta_M$ is $2$-Moufang. Let $M$ be such a difference matrix, in particular $M$ is Desarguesian. Up to replacing $M$ by an equivalent matrix, we can assume that each column of $M$ is equal to $D$ as a set (by Lemma~\ref{lemma:equivalence} (iv)). Moreover, up to permuting the lines of $M$ (see Lemma~\ref{lemma:equivalence} (i)), we can assume that the first column of $M$ is exactly $(d_1, \ldots, d_{q+1})^T$. So we look at matrices in
$$\mathcal{M} = \left\{M = \left(\begin{matrix}
d_1 & d_{\alpha_1(1)} & d_{\alpha_2(1)}\\
d_2 & d_{\alpha_1(2)} & d_{\alpha_2(2)}\\
\vdots & \vdots & \vdots \\
d_{q+1} & d_{\alpha_1(q+1)} & d_{\alpha_2(q+1)}\\
\end{matrix}\right) \suchthat
\begin{array}{c}
\alpha_1, \alpha_2 \in \Sym(q+1),\\
\Delta_M \text{ is $2$-Moufang}
\end{array}\right\}.$$
Let $M \in \mathcal{M}$ and write $M = (v_0, v_1, v_2)$. In the Desarguesian projective plane $\Pi_{v_t}$, a point is incident to $q+1$ lines, and the $q+1$ flags they form have $q+1$ different labels. The action of the point stabilizer on these $q+1$ flags thus gives a subgroup $G_t$ of $\Sym(q+1)$ which is conjugate to $\mathrm{P\Gamma L}(2,q)$. This subgroup $G_t \leq \Sym(q+1)$ does not depend on the chosen point because the subgroup of $\Aut(\Pi_{v_t})$ preserving the labels is transitive on points. The duality of $\Pi_{v_t}$ defined by $x \in P \mapsto -x \in L$, $x \in L \mapsto -x \in P$ also preserves the labels so the stabilizer of a line in $\Pi_{v_t}$ also gives birth to the same group $G_t \leq \Sym(q+1)$. We can moreover observe that $G_1 = \alpha_1^{-1} G_0 \alpha_1$ and $G_2 = \alpha_2^{-1} G_0 \alpha_2$, where $\alpha_1, \alpha_2 \in \Sym(q+1)$ behave as in the definition of $\mathcal{M}$. In $\Delta_M$, if $x_t$ and $x_{t'}$ are two adjacent vertices of type $t$ and $t'$ respectively, then the chambers adjacent to $x_t$ and $x_{t'}$ have the $q+1$ different labels, and Theorem~\ref{maintheorem:exotic'} exactly tells us that $\Delta_M$ is not $2$-Moufang when $G_t \neq G_{t'}$. Here we suppose that $\Delta_M$ is $2$-Moufang, so we deduce that $G_0 = G_1 = G_2$. As $\mathrm{P\Gamma L}(2,q)$ is its own normalizer in $\Sym(q+1)$, we obtain that $\alpha_1, \alpha_2 \in G_0$. In particular, we have $|\mathcal{M}|\ \leq |\mathrm{P\Gamma L}(2,q)|^2\ = (q(q^2-1) \eta)^2$. But Lemma~\ref{lemma:equivalence} (ii),(iii) implies that each matrix in $\mathcal{M}$ is equivalent to at least $(3 \eta)^2$ matrices in $\mathcal{M}$, so $|\mathcal{M}/\sim\!\!|\ \leq \left(\frac{q(q^2-1)}{3}\right)^2$ (where $\sim$ is the equivalence relation). This concludes the proof.
\end{proof}

%\begin{proof}[Proof of Corollary~\ref{maincorollary:bound}]
%This follows from Theorem~\ref{maincorollary:bound'} and Lemma~\ref{lemma:nMoufang}.
%\end{proof}

\begin{proof}[Proof of Corollary~\ref{maincorollary:limit}]
By \cite{Witzel}*{Theorem~B}, the number of isomorphism classes of Singer cyclic lattices with parameter $q = p^\eta$ is bounded below by $A(q) = \frac{1}{162\eta^3}((q+1)!)^2$. Moreover, by Corollary~\ref{maincorollary:bound} at most $B(q) = \left(\frac{q(q^2-1)}{3}\right)^2$ of them are non-exotic. The conclusion follows from the fact that $\frac{B(q)}{A(q)} \to 0$ when $q$ goes to infinity.
\end{proof}

%%%%%%%%%%%%%%%%%%%%
%%% BIBLIOGRAPHY %%%
%%%%%%%%%%%%%%%%%%%%

\end{document}